\documentclass[a4paper]{amsart}

\usepackage{amssymb}
\usepackage{soul}
\usepackage[ top = 3.5cm, bottom = 3.5cm, left = 3.75cm, right = 3.75cm]{geometry}

\usepackage{mathrsfs} 
\usepackage{color} 
\usepackage{subdepth} 
\usepackage{hyperref}

\usepackage{ifpdf}
\ifpdf
  \usepackage[all,pdf]{xy} 
\else
  \input xy 
  \xyoption{all}
  \xyoption{2cell} 
  \xyoption{v2}
\fi

\SelectTips{cm}{11} 

\DeclareMathOperator{\tr}{tr}

\DeclareMathOperator{\id}{id}

\DeclareMathOperator{\Map}{Map}
\DeclareMathOperator{\Ad}{Ad}

\DeclareMathOperator{\pr}{pr}
\DeclareMathOperator{\ev}{ev}

\renewcommand{\sec}{\epsilon_s}   

\newcommand{\scc}{c_s}   

\newcommand{\ec}{\epsilon}  

\theoremstyle{plain}
\newtheorem{theorem}{Theorem}[section]

\newtheorem{lemma}[theorem]{Lemma}
\newtheorem{proposition}[theorem]{Proposition}

\theoremstyle{definition}
\newtheorem{definition}[theorem]{Definition}

\theoremstyle{remark}
\newtheorem{example}{Example}[section]

\newtheorem{remark}{Remark}[section]

\numberwithin{equation}{section}
\numberwithin{figure}{section}

\newcommand{\cH}{{\mathcal H}}
\newcommand{\cG}{\mathcal{G}}

\newcommand{\cA}{{\mathcal A}}

\newcommand{\cK}{{\mathcal K}}

\newcommand{\cB}{{\mathcal B}}

\newcommand{\sC}{{\mathscr{C}}}

\newcommand{\CC}{{\mathbb C}}
\newcommand{\RR}{{\mathbb R}}
\newcommand{\ZZ}{{\mathbb Z}}

\renewcommand{\a}{\alpha}
\renewcommand{\b}{\beta}

\renewcommand{\d}{\delta}

\newcommand{\norm}[1]{\lVert {#1} \rVert}

\DeclareMathOperator{\String}{String}

\newcommand\longrightthreearrow{%
        \mathrel{\vcenter{\mathsurround0pt
                \ialign{##\crcr
                        \noalign{\nointerlineskip}$\longrightarrow$\crcr
                        \noalign{\nointerlineskip}$\longrightarrow$\crcr
                        \noalign{\nointerlineskip}$\longrightarrow$\crcr
                }%
        }}%
}

\newcommand{\xm}{\textsf{xm}}

\newcounter{anote}

\begin{document}

\title[Equivariant bundle gerbes]{Equivariant bundle gerbes}
  \author[M.\ K.\ Murray]{Michael K.\ Murray}
  \address[Michael Murray]
  {School of Mathematical Sciences\\
  University of Adelaide\\
  Adelaide, SA 5005 \\
  Australia}
  \email{michael.murray@adelaide.edu.au}

  \author[D.\ M.\ Roberts]{David Michael Roberts}
  \address[David M.\ Roberts]
{School of Mathematical Sciences\\
University of Adelaide\\
  Adelaide, SA 5005 \\
  Australia}
  \email{david.roberts@adelaide.edu.au}

\author[D.\ Stevenson]{Danny Stevenson}
\address[Danny Stevenson]
{School of Mathematical Sciences\\
University of Adelaide\\
  Adelaide, SA 5005 \\
  Australia}
  \email{daniel.stevenson@adelaide.edu.au}

  \author[R.\ F.\ Vozzo]{Raymond F.\ Vozzo}
\address[Raymond Vozzo]
{School of Mathematical Sciences\\
University of Adelaide\\
  Adelaide, SA 5005 \\
  Australia}
\email{raymond.vozzo@adelaide.edu.au}

\thanks{This research was supported under Australian
Research Council's {\sl Discovery Projects} funding scheme (project numbers DP120100106 and DP130102578).}


\begin{abstract} We develop the theory of   simplicial extensions for  bundle
gerbes and  their characteristic classes  with a view towards studying descent problems and equivariance for bundle gerbes. 
{Equivariant bundle gerbes are important in the study of orbifold sigma models}.  We consider in detail two 
examples: the basic bundle gerbe on a unitary group and a string structure for a principal bundle.  We show that the 
basic bundle gerbe is equivariant for the conjugation action  and calculate its characteristic  class; we show also that a string 
structure gives rise to a bundle gerbe which is equivariant for a natural action of the String 2-group. 

\end{abstract}
\maketitle

\tableofcontents

\section{Introduction}

Bundle gerbes were introduced by Murray in \cite{Mur} as a simpler alternative to the gerbes with band $U(1)$ described in the monograph \cite{BryBook} of Brylinski.  Bundle gerbes on a manifold $M$ are generalizations of the notion of line bundle on $M$: this fact is manifested in several ways, one of which is the existence of a characteristic class in $H^3(M,\ZZ)$, generalizing the Chern class of a line bundle.  In another direction, the notion of bundle gerbe allows for a particularly simple theory of connections and curving leading to a differential form representative for this characteristic class, the {\em 3-curvature} of the bundle gerbe connection and curving.  Crucially for applications to physics, bundle gerbes give rise to a notion of holonomy over a closed surface, generalizing the notion of the holonomy of a line bundle with connection around a loop.

There are many key examples of bundle gerbes with their origin in physical problems; for example in the study of anomalies in quantum field theory \cite{CarMicMur, CarMurMic2, HMSV}; (together with the allied notion of bundle gerbe module) in the study of $D$-brane charges in string theory \cite{BCMMS,Ho,K,MM,MW}; in the study of Chern--Simons theory \cite{CarJohMur, Mick} and its relation to string structures \cite{Wal}; and recently in the study of topological insulators \cite{CarDeletal, Gaw2}. The example which serves to motivate this paper is the role that bundle gerbes and bundle gerbe connections play in 2D sigma models with a Wess--Zumino term in the action functional.  The fields in such a theory are taken to be maps $\phi\colon \Sigma \to M$, where $M$ is the target manifold of the theory which is equipped with a closed 3-form $H$.  Locally, the Wess--Zumino term $S_{W\!Z}(\phi)$ is described by integrals over $\Sigma$ of $\phi^*B$, where $B$ is a local 2-form on $M$ solving the equation $dB = H$---the so-called {\em $B$-field}. In topologically non-trivial situations there are ambiguities which arise from the various choices that must be made in such a construction.  These ambiguities were analyzed by Gawedzki in \cite{Gaw} using the theory of Deligne cohomology, a certain hybrid of \v{C}ech and de Rham cohomology.  If one interprets the closed 3-form $H$ as the field strength or 3-curvature of a bundle gerbe with connection on $M$, this analysis can be carried out much more systematically and succinctly leading to an identification of the Feynmann amplitudes $\exp(iS_{W\!Z}(\phi))$ in terms of the holonomy of this bundle gerbe with connection \cite{CarMurMic2}.

This point of view is particularly well-adapted to the study of Wess--Zumino--Witten sigma models in which the target space is a compact Lie group $G$.  In particular, when $G$ is simple and simply connected, there is a canonical bundle gerbe with connection on $G$, the so-called {\em basic} bundle gerbe on $G$.  The case where the target manifold is a non-simply connected Lie group arising as the quotient of the simply connected cover $\tilde{G}$ by a finite subgroup $\Gamma$ of the center leads naturally to the notion of a $\Gamma$-equivariant gerbe on $\tilde{G}$. More generally, one can consider the notion of a $\Gamma$-equivariant gerbe on a manifold $M$ upon which $\Gamma$ acts; such an equivariant gerbe amounts to the notion of an ordinary gerbe on the orbifold $M/\Gamma$.  These equivariant gerbes can be used to give a similar description of the Wess--Zumino term when the target manifold is such an orbifold.  A natural question is how to extend this theory beyond the case of orbifolds, to the case where a compact Lie group acts smoothly on the manifold $M$.  This is the motivation for the present work which develops the theory of equivariant bundle gerbes; that is, we have a  bundle gerbe $\cG = (P, Y)$  over a manifold where $Y \to M$ is a surjective submersion and a Lie group $G$ acts smoothly on the right of $M$, and we want to investigate in what way this group action can be lifted to $\cG$.  This question has been studied in \cite{Bry, Cha} for bundle gerbes described by local data over an open cover of $M$, as well as in \cite{Ben, Gom} and notably in the general context of higher geometry in \cite{NikSch}. Our approach is to be contrasted with abstract approaches using higher categories in that one often wants, for the purposes of geometry and physics, specific manifolds and explicit descriptions of geometric objects (such as differential forms).

\subsection{Equivariance and simplicial extensions}

A convenient way of studying equivariant bundle gerbes is to use the theory of simplicial manifolds. To see why this is the case, and to motivate our constructions below, consider first the simpler case of an equivariant $U(1)$-bundle $P \to M$.  Then a right $G$ action on $P$ is a family of 
bundle maps $\phi_g \colon P \to P$, each covering the action of the corresponding $g \in G$ and satisfying $\phi_g \phi_h = \phi_{gh}$.  Because all our objects are smooth we would like the bundle
maps $\phi_g$ to depend smoothly on $g$ and a simple way to do that is to introduce the manifold
$ M \times G$ and two maps $d_0, d_1 \colon M \times G \to M$ defined by $d_0(m, g) = mg$ and $d_1(m, g) = m$. Then the bundle $d_0^{-1}(P) \otimes d_1^{-1}(P)^*$ has fibre at $(m, g)$ given by $P_{mg} \otimes P^*_m$ and the bundle maps $\phi$ can all be combined to give  a section of  $d_0^{-1}(P) \otimes d_1^{-1}(P)^*$ whose value at $(m, g) $ is $(\phi_g)_m(p) \otimes p^*$ where $p \in P_m$.  The condition that $\phi_g \phi_h = \phi_{gh}$ now becomes an equation on $M \times G^2$.  Returning to the case of bundle gerbes, it is natural to replace the idea of an isomorphism with a stable isomorphism and then the
condition  $\phi_g  \phi_h = \phi_{gh}$ may not hold exactly but rather up to a map $c_{g,h}$ between the stable  isomorphisms $\phi_g \phi_h$ and $\phi_{gh}$. In this case there is  
a coherence condition on the isomorphisms: $c_{g,h}c_{gh, k} = c_{g, hk} c_{h,k}$, which lives over $M \times G^3$. The manifolds $M, M\times G, M \times G^2 , \dots$ form a simplicial manifold---the nerve of the action groupoid---which we discuss further in Section \ref{sec:ss&bg}.

Now that we are in the setting of simplicial manifolds it becomes natural to generalise the idea above and formulate a notion of {\em simplicial extension}. In the simplest form this starts with a simplicial manifold $X_\bullet$ and a bundle gerbe 
$\cG = (P, Y)$ over $X_0$.  The definition of a simplicial extension then mimics the equivariance condition above.  We leave the detail for discussion in Section \ref{S:simplicial extensions} but note here some geometric consequences.  Firstly, given a simplicial manifold $X_\bullet$ there is an infinite-dimensional space $\| X_\bullet \|$, called the fat geometric realisation of $X_\bullet$, which contains a copy of $X_0$. Roughly speaking the existence of a simplicial extension is equivalent to the existence of an extension of the bundle gerbe $\cG$ from $X_0$ to $\| X_\bullet \|$. We do not prove this fact here but it motivates the choice of name. Secondly, we can realise the real cohomology of $\| X_\bullet \|$ in terms of de Rham classes on the various $X_k$ and this is denoted by $H^n(X_\bullet, \RR)$. There is a natural map 
 $$
 H^n(X_\bullet, \RR) \to H^n(X_0, \RR)
 $$
  for every $n\geq 0$ corresponding to the pullback from $\| X_\bullet \|$ to $X_0$.  A simplicial extension  of $\cG$ defines a class in $ H^3(X_\bullet, \RR)$, which we call the extension class of the simplicial extension, and this  maps to the real Dixmier--Douady class of the bundle gerbe $\cG$ in $H^3(X_0, \RR)$. 
  
By working with simplicial manifolds we can also consider the descent problem for bundle gerbes. This has been considered for bundle gerbes described by local data on an open cover of a manifold in \cite{Mei,Ste2000}. If $M \to N$ is a surjective submersion and $\cG$ a bundle gerbe on $M$ then the existence of descent data for $\cG$ is precisely the condition for $\cG$ to descend to a bundle gerbe on $N$.  We show in Section \ref{S:simplicial extensions} that such descent data is exactly a simplicial extension for the natural simplicial manifold $M, M^{[2]}, M^{[3]}, \ldots,$ where $M^{[k]}$ is the $k^{\text{th}}$ fibre product of $M$ with itself over $N$.
  This result is of interest in its own right but also important in understanding the descent of 
  equivariant bundle gerbes when the action of $G$ on $M$ arises from a principal $G$-bundle $M\to N$.
There are two natural notions of group action on a bundle gerbe; there is a {\em strong action} \cite{Gom2, MatSte, Mei}, where the group action on $M$ lifts to $Y \to M$ and also to $P \to Y^{[2]}$ and commutes with the bundle gerbe product; there is also the notion of \emph{weak action}, which corresponds to the general simplicial extension setting where essentially the group acts on $\cG$ by stable isomorphisms.  In Section \ref{S:eq} we show that for both strong and weak $G$ actions 
on a bundle gerbe $\cG$ over the total space of a principal $G$-bundle $M \to N$ there is a natural 
notion of quotient or descended gerbe on $N$.  In addition we show that a strong action induces a weak action and that the corresponding quotients agree, up to a specified stable isomorphism.

In \cite{MurSte} two of the authors gave a construction of the basic bundle gerbe $\cB_n$ on a unitary group $U(n)$.  In that work we discussed the fact that the conjugation action on $U(n)$ lifted
to a strong action of $U(n)$ on $\cB_n$.  In Theorem \ref{thm:equivariant gerbe on U(n)} we construct the 
extension class of this action and note that, in particular, it is non-trivial even in the case of $U(1)$ where $\cB_1$ and the conjugation action are both trivial. 

\subsection{2-group actions}

We also consider the case of an action of a 2-group on a manifold. This statement will need some unpacking. Firstly, a 2-group\footnote{not a $p$-group for $p=2$!} is a monoidal groupoid such that for each object of the groupoid there is another object that is an inverse, possibly only up to isomorphism. For the purposes of this article we will only introduce \emph{strict} 2-groups, where associativity holds, inverses are honest inverses and so on. This allows us to use the equivalent but less complicated crossed modules. Also, we are interested in using not just bare groupoids, but Lie groupoids, and so \emph{Lie} 2-groups. Many known Lie 2-groups, and the ones used in this article, arise as \emph{2-group extensions} of ordinary Lie groups. In our case, we take a Lie group $G$ with certain properties, and consider the String 2-group, which fits into an extension
\[
  \mathbf{B}U(1) \to \String_G \to G
\]
for a certain uncomplicated 2-group $\mathbf{B}U(1)$.  A lift of the structure group of a principal $G$-bundle $P$ to the group $\String_G$ is called a \emph{string structure} on $P$. These were first considered by Killingback in the context of heterotic string theory in \cite{Kil} (see also \cite{MurSte03, Wit}). The topology and geometry of string structures is also important in Witten's famous paper on the Dirac operator on loop spaces \cite{Wit2} and in Stolz and Teichner's program on elliptic cohomology \cite{ST}.

It is not difficult to ask for an action of a 2-group on a manifold (all 2-groups will be Lie 2-groups from now on) and it follows from the definition that such an action for $\String_G$ factors through the map to $G$. The reverse also holds: given a $G$ action, we can induce an action of $\String_G$.

This, then, is the context in which we look at bundle gerbes that are equivariant under the action of the 2-group $\String_G$ on a manifold. While the action factors through $G$, and so may appear uninteresting, an analogue of the discussion above for ordinary group actions becomes much more complicated; here the use of simplicial manifolds and simplicial extensions comes into its own. We shall leave the details for Section 6, but what we do is consider a \emph{string structure} for a principal $G$-bundle $P$, which can be given by a bundle gerbe on $P$ and some extra data. This bundle gerbe is \emph{not} $G$-equivariant, but it is $\String_G$-equivariant as we shall see in Theorem \ref{th:string}.

\subsection{Summary}

We start in Section 2 with a review of bundle gerbes and various simplicial objects that we need in our subsequent discussion. In Section 3 we present the general definition of our basic notion of a 
{\em simplicial extension} of a bundle gerbe. We present a number of examples and define the simplicial class of a simplicial extension.  Our first application uses the notion of simplicial extension to define a general descent condition for bundle gerbes $\cG$ over $M$ where $M \to N$ is a surjective submersion.  Our second application of simplicial extensions  in Section 4 is to define the notion of \emph{weak group action} on a bundle gerbe.  We show how it relates to the more obvious concept of \emph{strong group  action} and use the idea of descent to define the quotient of a bundle gerbe by a strong or weak group action. 
We also define equivariant classes for  strong and weak group actions. In Section 5 we consider the basic bundle gerbe on a unitary group defined by the first and third authors in \cite{MurSte} and show that it is  strongly equivariant under the conjugation action of $U(n)$ on itself. We give an equivariant connective structure and use this to calculate its strongly equivariant class, which is non-trivial even in the case of $U(1)$. 
Section 6 starts with some preliminary material on {crossed modules} and {bundle 2-gerbes}. We then show that a string structure for a principal $G$-bundle $P$, viewed as a trivialisation of the Chern--Simons bundle 2-gerbe of $P$, gives rise to a natural simplicial extension, meaning that it is equivariant for the natural action of $\String_G$ on $P$.


\section{Background on simplicial manifolds and bundle gerbes}
\label{sec:ss&bg}

\subsection{Simplicial manifolds}\label{SS:ss}

We recall some facts about simplicial objects in a category $\sC$ (see for example \cite{Bott, Dup, GoerJar}). We will mostly be interested in the category of smooth manifolds. Let $\triangle$ be the simplex category, whose objects are the finite ordinal sets $[0] = \{0\}, [1] = \{0, 1 \}, \dots$ and whose morphisms are order-preserving maps. A \emph{simplicial object in $\sC$} is a contravariant functor from $\triangle$ to $\sC$. A morphism between two simplicial objects is a natural transformation between the two functors defining them.

In more concrete terms, for the category of smooth manifolds, a simplicial object (i.e.\  a \emph{simplicial manifold}) is a sequence of manifolds $X_0, X_1, X_2, X_3, \dots$ together with maps $\alpha^*\colon X_j \to X_i$ for every arrow $\a\colon [i ]\to [j]$ in $\triangle$, 
satisfying the compatibility condition $\beta^*\alpha^* = (\alpha\beta)^*$. It is a standard fact that these can all be written in terms of a certain collection of maps $d_i \colon X_{p} \to X_{p-1} \ (i = 0, \ldots p)$ and  $s_i\colon X_p\to X_{p+1} \ (i = 0, \ldots p)$ called \emph{face} and \emph{degeneracy} maps, respectively, and satisfying the so-called \emph{simplicial identities} (see \cite{Dup}). Sometimes, only the face maps of a simplicial object will be important for us and we can ignore 
the degeneracies.  In such a case we will speak of a {\em semi-simplicial} object, eg.\  a semi-simplicial 
manifold. The face map $d_k \colon X_{p} \to X_{p-1}$ corresponds to the map $[{p-1}] \to [p]$ whose image does not contain $k$. We will typically denote a simplicial manifold $X_0, X_1, X_2, \dots$ by $X_\bullet$. A morphism of simplicial manifolds $Y_\bullet \to X_\bullet$ consists of a sequence of maps $Y_k \to X_k$ commuting with the 
face and degeneracy maps.

The following examples will be useful throughout the paper.

\begin{example}\label{ex:constant}
Let $X$ be a manifold. We define $X^{(\bullet)}$ to be the constant simplicial manifold with all face and degeneracy maps equal to the identity. Notice that if $X_\bullet$ is a simplicial manifold, then there is a map $X_0 \to X_k$ corresponding to the unique map $[k] \to [0]$ and this gives rise to a simplicial map $X_0^{(\bullet)} \to X_{\bullet}$.
\end{example}

\begin{example}\label{ex:cartesian}
Let $X$ be a manifold. Define $X^{\bullet+1}$ by $X^{k+1} = \Map ([k], X)$, with the simplicial maps $X^{i} \to X^{j}$ given by pullback by $[j] \to [i]$. Notice that $X^{k+1}$ is the cartesian product of $X$ and the face maps are given by omitting factors.
\end{example}

\begin{example}\label{ex:fibre prod}
Let $Y \to X$ be a submersion and let $Y^{[k]}$ be the fibre product of $k$ copies of $Y$. This defines a simplicial manifold $Y^{[\bullet+1]}$, where the simplicial maps are induced by restricting those of the cartesian product $Y^{\bullet + 1}$.

\end{example}

\begin{example}\label{ex:EG(M)}
If $M $ is a  manifold on which a Lie group $G$ acts smoothly we define a simplicial manifold $EG(M)_\bullet $ by $EG(M)_n = M \times G^n  $ for $n \geq 0$. The 
face maps are: 
   \begin{align*}
\label{eq:identities}
     d_k( m, g_1, \dots, g_n)  &= 
     \begin{cases}
 (mg_1, g_2, \dots, g_n)    & k = 0 \\
 (m, g_1, \dots, g_k g_{k+1}, \dots, g_n) & k = 1, 2, \dots, n-1\\
 (m, g_1, \dots,  g_{n-1}) & k = n.\\
      \end{cases}
      \end{align*}
In particular $X_1 =  M \times G$ and $X_0 = M$ and the two face maps $X_1 \to X_0$ are $d_0(m, g) = mg$ and $d_1(m, g) = m$.
\end{example}

In the case that $M \to N$ is a principal $G$-bundle then the simplicial manifolds in Examples \ref{ex:fibre prod} and \ref{ex:EG(M)} are isomorphic:

\begin{lemma}\label{L:EG(M)=M^[bullet]}
If  $M \to N$ is a $G$-bundle  then $EG(M)_\bullet \simeq M^{[\bullet+1]}$.
\end{lemma}
\begin{proof}
The isomorphism is given by maps $EG(M)_n  \to M^{[n]}$ by
$$
(m, g_1, g_2, \dots, g_n ) \mapsto (m, mg_1, mg_1g_2, \dots, mg_1g_2 \dots g_n).
$$
It can be easily checked that this defines a simplicial map.
\end{proof}

\begin{example}\label{ex:EK(M)}
If $M$ is a manifold and $\cK$ is a crossed module (Definition~\ref{D:crossed module}) that acts on $M$ 
(Definition~\ref{D:crossed module action}), then there is a simplicial manifold $E \cK (M)_\bullet$ similar to the one defined in Example \ref{ex:EG(M)} for a Lie group. We will use this simplicial manifold in Section \ref{S:string}, where we will give a precise definition.
\end{example}

We call a simplicial object in the category of surjective submersions a \emph{simplicial surjective submersion}. Explicitly, this will be a pair of simplicial manifolds $Y_\bullet$ and $X_\bullet$ and a simplicial map $Y_\bullet \to X_\bullet$ with the 
property that $Y_k \to X_k$ is a surjective submersion for all $k\geq 0$.  Note that simplicial surjective 
submersions are preserved under pullback, in the sense that if $Y_\bullet\to X_\bullet$ 
is a simplicial surjective submersion, then so is the induced map $Y_\bullet\times_{X_\bullet} Z_\bullet\to Z_\bullet$ 
for any simplicial map $Z_\bullet\to X_\bullet$.  

Let $X_\bullet$ be a simplicial manifold and $Y \to X_k$ a surjective submersion. Define $\delta(Y) \to X_{k+1}$
by
$$
\delta(Y) = d_0^{-1}(Y) \times_{X_{k+1}} d_1^{-1}(Y) \times_{X_{k+1}} d_2^{-1}(Y) \times_{X_{k+1}} \cdots 
$$
Denote by  $d_i \colon \delta(Y) \to Y$ the obvious projections covering the face maps $d_i \colon X_{k+1} \to X_k$.  We can define inductively a 
family $\delta^k(Y)$ with  maps $\delta^k(Y) \to \delta^{k-1}(Y)$.  We remark that $\delta^{\bullet + 1}(Y)$
is not a simplicial manifold. However, in the next example we discuss a related construction, which is a simplicial surjective submersion.

\begin{example}\label{ex:mu inverses1}
If $X_\bullet$ is a simplicial manifold then for each $i = 0, \dots, p$ there are maps $\mu_i \colon X_p \to X_0$ induced by the inclusion $[0] \to [p]$ defined by  $0 \mapsto i$. Then 
$$
\mu = (\mu_0, \dots, \mu_k)  \colon X_k \to X_0^{k+1}
$$
defines a morphism of simplicial manifolds $\mu \colon  X_\bullet \to X^{\bullet+1}_0$. If $Y \to X_0$ is a surjective submersion then so is $Y^{k+1} \to X_0^{k+1}$ and we can define a simplicial surjective submersion $\mu^{-1}(Y^{\bullet+1}) \to X_\bullet$. We will use this example extensively in the rest of the paper.
\end{example}

We will use the following similar construction in Section \ref{S:string} when we discuss bundle 2-gerbes.

\begin{example}\label{ex:mu inverses2}
We have maps $\mu_{ij} \colon X_p \to X_1$ for $0 \leq i < j \leq p$ induced by the map $[1] \to [p]$ defined by $0 \mapsto i$ and $1 \mapsto j$. We can assemble these into maps $X_k \to X_1^{k(k+1)/2}$, which we also call $\mu$. If $Y \to X_1$ is a surjective submersion then as above we can pull back $Y^{k(k+1)/2}$ to give a surjective submersion over $X_k$, for $k \geq1$. So we have a collection of manifolds
$$
\mu^{-1}(Y)_k = \begin{cases} X_0, & k = 0\\ Y, & k=1\\ \mu^{-1}(Y^{k(k+1)/2}), & k >1 \end{cases}
$$
and maps $\mu^{-1}(Y)_k \to \mu^{-1}(Y)_{k-1}$ satisfying the simplicial identities for face maps. Note that unlike Example \ref{ex:mu inverses1}, in general $\mu^{-1}(Y)_\bullet$ is only a {\em semi}-simplicial manifold, since we may not have degeneracy maps $\mu^{-1}(Y)_0 = X_0 \to \mu^{-1}(Y)_1 = Y$.  The canonical map $\mu^{-1}(Y)_\bullet \to X_\bullet$ is then a semi-simplicial surjective submersion.  
\end{example}

It is clear that one could continue this and define a semi-simplicial surjective submersion given a surjective submersion $Y \to X_k$ as in the paragraph preceding Example \ref{ex:mu inverses1}, however we will only need these two cases.

\begin{lemma}\label{L:Y to delta Y}
Let $Y_\bullet \to X_\bullet$ be a  simplicial surjective submersion. For any $k$ there is a map of surjective submersions $ Y_{k+1} \to \delta(Y_k)$ covering projections to $X_{k+1}$ such that 
each composition $Y_{k+1} \to \delta(Y_k) \to Y_k$ is the corresponding face map $Y_{k+1} \to Y_k$.
\end{lemma}
\begin{proof}
We define $ Y_{k+1} \to \delta(Y_k)  $ by 
$ y \mapsto  (d_0(y), d_1(y),  \cdots )$ and the result follows.
\end{proof}

For simplicity, throughout this paper we will only work with a restricted notion of morphism between simplicial surjective submersions. Namely, if  $Y'_\bullet \to X_\bullet$ and $Y_\bullet \to X_\bullet$ 
are  simplicial surjective submersions then a morphism 
$$
\phi_\bullet \colon (Y'_\bullet \to X_\bullet) \to  (Y_\bullet \to X_\bullet)
$$
is a family of maps $\phi_\bullet \colon Y'_\bullet \to Y_\bullet$ such that every diagram 
$$ 
\xymatrix{ 
Y_k' \ar[d] \ar[r]^{\phi_k} & Y_k \ar[d] \\ 
X_k \ar@2{-}[r]  & X_k } 
$$ 
commutes for every $k \geq 0$.

 In particular we have 
 
 \begin{lemma}
 \label{lemma:uni-sss}
 If $Y_\bullet \to X_\bullet$ is a simplicial surjective submersion there is 
 a morphism of simplicial surjective submersions 
$$
\mu_\bullet \colon (Y_\bullet \to X_\bullet) \to (\mu^{-1}(Y_0^{\bullet+1}) \to X_\bullet).
$$
\end{lemma}
\begin{proof}
For every $k=0, 1, \dots$ we have a commuting diagram
$$ 
\xymatrix{ 
Y_k \ar[d] \ar[r]^-\mu & Y^{k+1}_0 \ar[d] \\ 
X_k \ar[r]^-\mu & X^{k+1}_0 } 
$$ 
and the result follows from this.
\end{proof}

\subsection{Bundle gerbes and simplicial manifolds}\label{SS:ss&bg}

Let  $X_\bullet$ be a simplicial manifold and $Q \to X_p$ be a $U(1)$-bundle for some $p\geq 0$. We define a new 
$U(1)$ bundle $\delta(Q)$ on $X_{p+1}$ by 
$$
\delta(Q) = d_0^{-1}(Q) \otimes d_1^{-1}(Q)^* \otimes d_2^{-1}(Q) \otimes \cdots 
$$
For $i \neq j$ let $\pi_{ij} \colon X_{p+2} \to X_{p}$ denote the map 
induced by the unique order-preserving map $[{p}] \to [{p+2}]$ whose image does not 
contain $i$ and $j$.  Notice that if $i \leq  j$ then we have $d_i d_j = \pi_{i (j+1)}$ 
and if $i > j$ then $d_i d_j  = \pi_{ij}$. It follows easily that there is an isomorphism 
\begin{equation}\label{E:canon triv of delta^2}
\delta^2(Q) =\bigotimes_{ 0 \leq i < j \leq p+2 }  \pi_{ij}^{-1}(Q \otimes Q^*) 
\end{equation}
and hence has $\delta^2(Q)$ has a canonical trivialisation. Explicitly, we may define a section $c$ of $\d^2(Q)$ 
whose value at $x\in X_{p+2}$ is 
\begin{equation*}\label{E:canon section of delta^2}
c(x) = \bigotimes_{ 0 \leq i < j \leq p+2 }  \pi_{ij}^{-1}(q_{ij} \otimes q_{ij}^*), 
\end{equation*}
for some choice of elements $q_{ij}$ in the fibres over $\pi_{ij}(x)$. We will usually denote this canonical section by $1$.

Recall the definition of a simplicial line bundle from \cite{BryMac}. 

\begin{definition}[\cite{BryMac}] 
\label{simplicial:bundle}
Let $X_\bullet$ be a simplicial manifold.  
A {\em simplicial line bundle}  over $X_\bullet$ is a pair  $(Q, \sigma)$ defined as follows:
  \begin{enumerate}
  \item $Q \to X_1$ is a $U(1)$-bundle;
  \item $\sigma$ is a section of $\delta(Q) \to X_2$ such that  $\delta(\sigma) = 1 \in \delta^2(Q)$.
    \end{enumerate}
\end{definition}

If $Y \to M$ is a surjective submersion there is an  equivalence between bundle gerbes $\cG$ over $M$ and simplicial line bundles over  $Y^{[\bullet+1]}$.   To see this note that if $P$ is a simplicial line bundle over $Y^{[\bullet+1]}$ then $P \to Y^{[2]}$ is a $U(1)$-bundle and $\sigma(y_1, y_2, y_3)$ 
is an element of 
$$
P_{(y_2, y_3)} \otimes P^*_{(y_1, y_3)} \otimes P_{(y_1, y_2)}
$$
which must be of the form $b \otimes m(a, b)^* \otimes a$ for some bundle morphism
$$
m \colon P_{(y_1, y_2)} \otimes P_{(y_2, y_3)} \to P_{(y_1, y_3)}
$$
and any elements $a \in P_{(y_1, y_2)}$ and $b \in P_{(y_2, y_3)}$. The morphism $m$ defines an associative
bundle gerbe multiplication if and only if $\delta(\sigma) = 1$. If $Y \to M$ is a surjective submersion then we denote the bundle gerbe given by the simplicial line bundle $P$ over $Y^{[\bullet + 1]}$ by $(P, Y)$.

We say that a simplicial line bundle $(Q, \sigma)$ over $X_\bullet$ is {\em trivial} if there is a $U(1)$-bundle $T \to X_0$
such that $Q = \delta(T)$ and $\sigma$ is the section $1$ of $\delta^2(T)$.  It is easy to see that a bundle
gerbe is trivial if and only if it is trivial when regarded as a simplicial line bundle. In this case we have a line bundle $T \to Y$ such that $\d(T) = P$, where here the $\d$ operation is for the simplicial manifold $Y^{[\bullet + 1]}$. For clarity, we will write $\d_Y(T)$ for this whenever there is any possibility of confusion.

Recall that if  $(P, Y)$ is a bundle gerbe on $M$ with trivialisations $T$ and $R$, then there is a canonical definition 
of a line bundle $L \to M$ with the property that $T = R \otimes \pi^*(L)$. The point is that we have an isomorphism $\delta(T) \to \delta(R)$ and hence descent data for $T \otimes R^*$. We denote the descended line bundle $L$ by $T \oslash R$.

We will also be interested in a particular class of trivial bundle gerbes: we say a bundle gerbe $(P, Y)$ over $M$ is {\em strongly trivial} if $P = Y^{[2]} \times U(1)$ and the multiplication is the {\em trivial} multiplication, in other words $(y_1, y_2, z_1)(y_2, y_3, z_2) = (y_1, y_3, z_1 z_2)$.
  Note that $Y \times U(1)$ is a trivialisation of a strongly trivial bundle gerbe and it 
follows that if $T$ is a trivialisation of a strongly trivial bundle gerbe then $T = T \oslash(Y \times U(1))$ descends to $M$. If $T$
is trivialised so that $T = Y \times U(1)$ we say that $T$ is a {\em strong trivialisation} if the induced trivialisation morphism  
$\delta_Y( Y \times U(1)) \to Y^{[2]} \times U(1)$ is the identity on the $U(1)$ factor.  We have the following Lemma
whose proof is straightforward.

\begin{lemma}
\label{lemma:strong-triv}
Assume that $(Q, Y)$ is strongly trivial over $M$ and trivialised by $T$ so that $T$ descends to $T \oslash (Y \times U(1)) \to M$ as discussed above.  Assume further that $s \colon Y \to T$ is a section so it induces an isomorphism $T \simeq Y \times U(1)$. Then $s$ descends 
to a section of $T \oslash (Y \times U(1)) \to M$ if and only if the isomorphism $T \simeq Y \times U(1)$ induces a strong trivialisation 
of $(Q, Y)$.
\end{lemma}

Let $(P, Y)$ be a bundle gerbe over $M$. If $f \colon N \to M$ is a map then 
the bundle gerbe pulls back to a bundle gerbe $(f^{-1}(P), f^{-1}(Y))$ over $N$. If $X \to M $ is also a
submersion and $f \colon X \to Y$ is a map of manifolds over $M$ then $(f^{-1}(P), X)$ is a bundle gerbe over $M$. 
Notice that both of these examples can be understood as the pullback of the simplicial line bundle over $Y^{[\bullet+1]}$ by $f \colon f^{-1}(Y^{[\bullet+1]}) \to Y^{[\bullet+1]}$ and $f \colon X^{[\bullet+1]} \to Y^{[\bullet+1]}$, respectively.

If  $T \to Y$ is a trivialisation of $(P, Y)$ then we denote the induced trivialisation of $(f^{-1}(P), f^{-1}(Y))$ by $f^{-1}(T)$.

Assume that $P$ has a bundle gerbe connection $\nabla$ 
and curving $f$.  Choose  a connection $\nabla_T $ for $T$ satisfying $\nabla = \delta_Y(\nabla_T)$. Then $F_{\nabla} = \delta_Y(F_{\nabla_T}) = 
\delta_Y(f)$ so that $F_{\nabla_T} - f = \pi^*\nu_T$ for some $\nu_T \in \Omega^2(M)$. If $R$ is another trivialisation with connection $\nabla_R$ such that $\nabla = \d_Y(\nabla_R) $ then, as above, $T \otimes R^*$ descends to a bundle $T \oslash R$. Then $\nabla_T - \nabla_R$ descends to a connection $\nabla_{T\oslash R}$ on $ T \oslash R $ whose curvature $F_{T\oslash R} = \nu_T - \nu_R$.

If $X_\bullet$ is a simplicial manifold and $(P, Y)$ a bundle gerbe over $X_k$,  where $Y \to X_k$ 
is a surjective submersion, we can define  a bundle gerbe $(\delta(P), \delta(Y))$
over $X_{k+1}$.   We will also be interested in a more complicated case.  Let $Y_\bullet \to X_\bullet$ 
be a simplicial surjective submersion and let $(P, Y_k)$ be a bundle gerbe over $X_k$.  Then by  Lemma \ref{L:Y to delta Y} we can restrict $(P, \delta(Y_k))$ to form $(P, Y_{k+1})$ over $X_{k+1}$. Further, we can repeat this process and form $(\d^2(P), Y_{k+2})$ over $X_{k+2}$. Notice that, as per (\ref{E:canon triv of delta^2}), we have that
$$
\d^2(P) = \bigotimes_{ 0 \leq i < j \leq p+4 }  \pi_{ij}^{-1}(P \otimes P^*),
$$
where we have used $\pi_{ij}$ to denote the induced map $Y_{k+2}^{[2]} \to Y_k^{[2]}$. Therefore we see that $(\d^2(P), Y_{k+2})$ is canonically isomorphic to the strongly trivial bundle gerbe $(Y_{k+2}^{[2]} \times U(1), Y_{k+2})$.

We have 

\begin{lemma}
\label{lemma:delta-squared-strong-triv}
  Assume that $(Q, Y_k)$ is a bundle gerbe over $X_k$ and $R \to Y_k$ is a trivialisation 
of $Q$. Then $(\delta^2(Q), Y_{k+2})$ is strongly trivialised by $\delta^2(R) \to Y_{k+2}$.
\end{lemma}
\begin{proof}
We have that
$$
\delta^2(Q) = \bigotimes_{0 \leq i, j, \leq k+4} \pi_{ij}^{-1}(  Q ) \otimes \pi_{ij}^{-1}(Q^*) \to Y_{k+2}^{[2]}
$$
and 
$$
\delta^2(R) = \bigotimes_{0 \leq i, j, \leq k+4} \pi_{ij}^{-1}(  R) \otimes \pi_{ij}^{-1}( R^* ) \to Y_{k+2}.
$$
If $ \psi \colon \delta_{Y_k}(R) \to Q$ is the trivialisation morphism, the induced trivialisation morphism 
$\delta_{Y_{k+2}}( \delta^2(T)) \to \delta^2(Q)$ is 
$$
\bigotimes_{0 \leq i, j, \leq k+4} \pi_{ij}^{-1}(  \psi) \otimes \pi_{ij}^{-1}( \psi^* ) \to Y_{k+2},
$$
which is the trivial morphism induced by the identity map on $U(1)$. 
\end{proof}

We will be particularly interested in the following examples of this: 

\begin{example}\label{ex:delta G with the mu inverse spaces}
If $Y \to X_0$ is a surjective submersion and $(P, Y)$ is a bundle gerbe, then we have the simplicial surjective submersion $\mu^{-1}(Y^{\bullet +1 }) \to X_\bullet$ from Example \ref{ex:mu inverses1} and we can form the bundle gerbe $(\delta(P), \mu^{-1}(Y^{2}))$ over $X_{1}$ and the bundle gerbe $(\delta^2(P), \mu^{-1}(Y^3))$ over $X_2$. 
\end{example}
\begin{example}\label{ex:delta G for a bundle 2-gerbe}
If $Y \to X_1$ is a surjective submersion and $(P, Y)$ is a bundle gerbe, then we have the semi-simplicial surjective submersion $\mu^{-1}(Y)_\bullet \to X_\bullet$ from Example \ref{ex:mu inverses2} and we can form the bundle gerbe $(\delta(P), \mu^{-1}(Y)_2)$ over $X_{2}$ and the bundle gerbe $(\delta^2(P), \mu^{-1}(Y)_3)$ over $X_3$. 
\end{example}

Finally, we make a remark about notation. We will be concerned with bundle gerbes $\cG = (P, Y)$ over simplicial manifolds and we shall be using the operation $\d$ repeatedly. As in Examples \ref{ex:delta G with the mu inverse spaces} and \ref{ex:delta G for a bundle 2-gerbe} and the discussion preceding Lemma \ref{lemma:delta-squared-strong-triv}, we will often be interested in the bundle gerbe $(\d(P), \d(Y))$ restricted to some subspace of $\d(Y)$. To make it clear precisely which bundle gerbe we mean by $\d(\cG)$, we will use the notation $(\d(P), Y_k)$ (where $Y_\bullet \to X_\bullet$ is a simplicial surjective submersion) whenever there is chance of confusion.

\subsection{Simplicial de Rham cohomology}
We recall the definition of the simplicial de Rham cohomology of a simplicial manifold $X_\bullet$ \cite{Dup}. 
Associated canonically to $X_\bullet$ is the bicomplex with differentials 
\begin{align*}
D_{p, q} \colon \Omega^p(X_q) &\to \Omega^{p+1}(X_q) \oplus \Omega^p(X_{q+1}) \\
\eta_{(p, q)} &\mapsto ((-1)^q d \eta_{(p, q)} , \d  \eta_{(p, q)} ).
\end{align*}
We combine these to form the total complex in the usual fashion: 
$$
D \colon \bigoplus_{p+q = r} \Omega^p(X_q) \to \bigoplus_{p+q = r+1} \Omega^p(X_q).
$$
The cohomology of this total complex is defined to be the simplicial de Rham cohomology, denoted $H^{r}(X_\bullet, \RR)$.
 Note that this is also the real cohomology of the fat realisation 
$\norm{X_{\bullet}}$ (see for instance Proposition 5.15 of \cite{Dup}).  
For later convenience we introduce the notation $\cA^*(X_\bullet)$ for this total complex. 

 Of particular interest
will be $H^3(X_\bullet, \RR)$  and we note that a class consists of 
$$
\eta = ( \eta_{(0, 3)}, \eta_{(1, 2)},  \eta_{(2, 1)}, \eta_{(3, 0)}  ) \in \Omega^0(X_3) \oplus \Omega^1(X_2) \oplus \Omega^2(X_1) \oplus \Omega^3(X_0) 
$$
satisfying
\begin{align*}
D \eta &= 
(\d \eta_{(0, 3)}, - d \eta_{(0, 3)} + \d \eta_{(1, 2)},  d \eta_{(1, 2)} + \d \eta_{(2, 1)}, 
 - d \eta_{(2, 1)} + \d \eta_{(3, 0)},  d \eta_{(3, 0)}) \\
 &= (0, 0, 0, 0, 0)
 \end{align*}
  up to addition of a cocycle of the form $D \rho$ so that 
\begin{multline*}
\eta + D \rho =    (\eta_{(0, 3)} + \delta \rho_{(0, 2)}, \eta_{(1, 2)}+ d \rho_{(0, 2)} + \d \rho_{(1, 1)}, 
\\ \qquad \eta_{(2, 1)} -d \rho_{(1, 1)}+ \d \rho_{(2, 0)}, \eta_{(3, 0)}+ d \rho_{(2, 0)}).
\end{multline*}

\section{Simplicial extensions}\label{S:simplicial extensions}

\subsection{Simplicial extensions of bundle gerbes}\label{SS:simplicial extensions}

Before we define the notion of a simplicial extension we need the following:

\begin{proposition}\label{prop:trivialisation descent}
Let $Y_\bullet \to X_\bullet$ be a simplicial surjective submersion and $\cG = (P, Y_k)$ be a bundle gerbe over $X_k$. Assume that $\delta(\cG) = (\delta(P), Y_{k+1})$ has a trivialisation  $T \to Y_{k+1}$. Then $\delta(T) \to Y_{k+2}$ descends to a line bundle $A_T \to X_{k+2}$ and the canonical trivialisation of $\d^2 (T) \to Y_{k+3}$ descends to a trivialisation of $\delta(A_T) \to X_{k+3}$. 
\end{proposition}
\begin{proof}
Notice first that the line bundle $\delta(T)$ trivialises the bundle gerbe $(\delta^2(P), Y_{k+2})$, which is canonically isomorphic to the strongly trivial bundle gerbe $(Y_{k+2}^{[2]} \times U(1), Y_{k+2})$, as we observed in the discussion preceding Lemma \ref{lemma:delta-squared-strong-triv}. Hence $\delta(T)$ descends to $A_T := \delta(T) \oslash (Y_{k+2} \times U(1)) \to X_{k+2}$.

To see that the canonical trivialisation of $\d^2(T) \to Y_{k+3}$ descends to $\d(A_T) \to X_{k+3}$ we apply Lemma \ref{lemma:delta-squared-strong-triv} to the bundle gerbe $(\delta(P), Y_{k+1})$ with trivialisation $T \to Y_{k+1}$ to deduce that $(\delta^3(P), Y_{k+3})$ is strongly trivialised by $\delta^2(T) \to Y_{k+3}$. Then Lemma 
\ref{lemma:strong-triv} implies that the section of $\delta^2(T)$ descends to a section of $\delta(A_T) \to X_{k+3}$.  
\end{proof}

Using this we can make the following definition:

  \begin{definition}     Let $X_\bullet$ be a simplicial manifold
  and $\cG = (P, Y_0)$ be a bundle gerbe over $X_0$. A
{\em simplicial extension} of  $\cG$ over $X_\bullet$ is a triple $(Y_\bullet, T, s)$ consisting of:
  \begin{enumerate}
  \item $Y_\bullet \to X_\bullet$  a simplicial surjective submersion;
  \item a trivialisation $T \to Y_1 $ of $\delta(\cG) = (\delta(P), Y_1)$ over $X_1$; and
  \item a section $s \colon X_2 \to A_T$ satisfying $\delta(s) =1$ relative to the canonical trivialisation of $\delta(A_T)$.     \end{enumerate}
\end{definition}

If $(Y_\bullet, T, s)$ is a simplicial extension then pulling back $s$ to a section of $\delta(T)$ defines a simplicial line bundle over $Y_\bullet$, which we denote by $[T, s]$. 

A simple but useful example of a simplicial extension is the following, defined for any bundle gerbe $(P, Y)$ over $X$.

\begin{proposition}
\label{prop:Psimplicial}
A bundle gerbe $(P, Y)$ over a manifold $X$ defines a simplicial extension $(Y^{[\bullet + 1]}, T, s)$ over the constant simplicial manifold $X^{(\bullet)}$ whose induced 
simplicial line bundle on $Y^{[\bullet+1]}$ is precisely $(P, Y)$. 
\end{proposition} 

\begin{proof}
The condition that $T$ trivialises $\delta(P)$ can be written as follows.  Let $(y_0,y_1) $, $(y_0',y_1')$, $(y_0'', y_1'') \in Y^{[2]}$, then there is an isomorphism 
\begin{equation}
\label{eq:T-triv-condition} 
T_{(y_0, y_1)} \simeq P_{(y_0, y_0')} \otimes  T_{(y_0', y_1')} \otimes  P_{(y_1', y_1)}
\end{equation}
and the isomorphism
$$
T_{(y_0, y_1)} \simeq P_{(y_0, y_0'')}  \otimes T_{(y_0'', y_1'')} \otimes  P_{(y_1'', y_1)}
$$
is equal to the induced isomorphism
\begin{align*}
T_{(y_0, y_1)} &\simeq P_{(y_0, y_0')} \otimes  T_{(y_0', y_1')} \otimes  P_{(y_1', y_1)} \\
              & \simeq P_{(y_0, y_0')} \otimes  P_{(y_0', y_0'')}  \otimes T_{(y_0'', y_1'')} \otimes  P_{(y_1'', y_1')}  \otimes P_{(y_1', y_1)} \\
  & \simeq P_{(y_0, y_0'')} \otimes T_{(y_0'', y_1'')}  \otimes P_{(y_1'', y_1)}            ,
\end{align*}
where we use the bundle gerbe multiplication to get from the second to the third line.

In this case, a trivialisation $T \to Y^{[2]}$ is given by $T_{(y_0, y_1)} = P_{(y_0, y_1)}$ and defining the map using the bundle gerbe product.  Then 
$$ 
\delta(T)_{(y_0, y_1, y_2)}  = P_{(y_1, y_2)} \otimes P_{(y_0, y_2)}^* \otimes P_{(y_0, y_1)} ,
$$
and a section is given by the bundle gerbe multiplication. It is easy to see that the descent equation preserves this section and 
this defines the required section $s$ of $A_T$.  The associativity of the bundle gerbe product gives us $\delta(s) = 1$. By construction this simplicial extension pulls back to the simplicial line bundle defined by $(P, Y)$.
\end{proof}

\begin{proposition}
\label{prop:simp-ext-pullback}
Let  $(Y'_\bullet \to X_\bullet)$  and $ (Y_\bullet \to X_\bullet)$ be simplicial surjective submersions where 
$(Y'_0 \to X_0) = (Y_0 \to X_0)$.  Assume we have a morphism $\phi \colon (Y'_\bullet \to X_\bullet) \to (Y_\bullet \to X_\bullet)$
which is the identity when $k = 0$. Then a simplicial extension $(Y_\bullet, T, s)$ of $(P, Y_0)$ over $X_\bullet$ pulls
back to a simplicial extension $(Y'_\bullet, \phi^{-1}(T), s)$ of $(P, Y_0)$ over $X_\bullet$. 
\end{proposition}
\begin{proof}
The trivialisation $T$ of $(\d(P), Y_1)$ pulls back to a trivialisation $\phi^{-1}(T)$ of the bundle gerbe $\phi^{-1}(\d(P), Y_1) = (\d(P), Y_1')$. Further, since $\phi$ is a simplicial map, $\d(\phi^{-1}(T)) = \phi^{-1}(\d(T))$ which descends to $A_T$. Hence the triple $(Y_\bullet', \phi^{-1}(T), s)$ is a simplicial extension of $(P, Y_0)$.
\end{proof}

More surprising is the following:

\begin{proposition}\label{prop:there exists a mu extension}
Let $(Y_\bullet \to X_\bullet)$ be a simplicial surjective submersion and $(Y_\bullet, T, s)$ be a simplicial 
extension of $(P, Y_0)$ over $X_\bullet$. Then there is a simplicial extension 
$$
(\mu^{-1}(Y_0^{\bullet+1}), \mu(T) , \mu(s))
$$
which pulls back to $(Y_\bullet, T, s)$ by the morphism in Proposition \ref{prop:simp-ext-pullback}.  
\end{proposition}
\begin{proof}
Notice first that over $X_1$ we have two bundle gerbes  $(\delta(P), \mu^{-1}(Y^2))$ and its pullback by $\mu \colon Y_1 \to \mu^{-1}(Y^2)$ which is 
  $(\mu^{-1}(\delta(P)), Y_1)$.  If $T \to Y_1$ is a trivialisation of $(\mu^{-1}(\delta(P)), Y_1)$ then from Proposition \ref{prop:descent-trivialisations} we know that there is an induced trivialisation $\mu(T) \to \mu^{-1}(Y^2)$ 
of $(\delta(P), \mu^{-1}(Y^2))$ which pulls back to $T \to Y_1$. The construction of $A_T$ in Proposition \ref{prop:trivialisation descent} depends only on $\delta(P)$ and $T$ and it follows that there are isomorphisms $A_T \simeq A_{\mu(T)}$ which 
commute with the trivialisations of $\delta(A_T)$ and $\delta(A_{\mu(T)})$.  Hence we can define a section $\mu(s)$ of $A_{\mu(T)}$ 
and $(\mu^{-1}(Y_0^{\bullet+1}), \mu(T) , \mu(s))$ is a simplicial extension.  By 
Proposition \ref{prop:descent-trivialisations} and the  construction it pulls back in the required manner.
\end{proof}

Proposition \ref{prop:there exists a mu extension} means we could simplify the definition of simplicial extension by always working with the simplicial surjective submersion $\mu^{-1}(Y_0^{\bullet + 1}) \to X_\bullet$. Indeed, when we specify a simplicial extension using only the pair $(T, s)$ then it will be understood that $Y_\bullet = \mu^{-1}(Y_0^{\bullet + 1})$.
We note however that as we shall see in the next section, in practice we find that it is useful to allow the extra flexibility of the choice of $Y_\bullet$.

Recall from Example \ref{ex:constant} that we have the map $X_0 \to X_k$ induced by the unique map $[k] \to [0]$ and that the composition $X_0 \to X_k \to X_0$ with each of the projections $\mu_i \colon X_k \to X_0$ is the identity. This means $\mu^{-1}(Y^{\bullet + 1}) \to X_\bullet$ pulls back to $Y^{[\bullet + 1]} \to X_0^{(\bullet)}$ and the simplicial line bundle $[T, s]$ pulls back to a simplicial line bundle on $Y^{[\bullet + 1]}$, which is a bundle gerbe on $X_0$. 
We have 

\begin{proposition}
\label{prop:Prestriction} 
Let $(T, s)$ be a simplicial extension of the bundle gerbe $(P, Y)$ over $X_\bullet$. 
The bundle gerbe on $X_0$ defined by the pullback of $[T, s]$ to $X_0^{(\bullet)}$ via the simplicial map $X_0^{(\bullet)} \to X_\bullet$ is isomorphic to $(P, Y)$. 
\end{proposition}
\begin{proof}
The pullback of the simplicial extension of $(P, Y)$ over the simplicial manifold $X_\bullet$ is 
a simplicial extension of $(P, Y)$ over $X_0^{(\bullet)}$.    We know from Proposition \ref{prop:Psimplicial} that 
$(P, Y)$ also defines a simplicial extension over $X_0^{(\bullet)}$. So we only need to prove that any two simplicial extensions of $(P, Y)$
over $X_0^{(\bullet)}$ are isomorphic.  Assume then that $(T_1, s_1)$ and $(T_2, s_2)$ are simplicial extensions of $(P, Y)$
over $X_0^{(\bullet)}$.  Notice that $T_1$ and $T_2$ are trivialisations of $\delta(P)$ so there exists a line
bundle $L \to X_0$ defined by $L = T_1 \oslash T_2$.  Then $\delta(L) = A_{T_1} \otimes A_{T_2}^*$. But 
as all the maps are the identity $\delta(L) = L$, which has a section defined by $s_1 \otimes s_2^*$.   Hence
$L$ is trivial which defines an isomorphism from $T_1$ to $T_2$. Moreover, by definition this isomorphism 
maps $s_1$ to $s_2$ and hence defines an isomorphism of simplicial line bundles. 
\end{proof}

 The same canonical simplicial map $X_0^{(\bullet)}\to X_{\bullet}$
 induces a homomorphism 
 $$
 H^n(X_\bullet, \RR) \to H^n(X_0, \RR)
 $$
  for every $n\geq 0$.

\begin{proposition}
\label{prop:simplicial-class}
A simplicial extension of $\cG$ over $X_\bullet$ defines a class in the simplicial de Rham cohomology $H^3(X_\bullet, \RR)$ which maps to the real Dixmier--Douady class 
of $\cG$ in $H^3(X_0, \RR)$.  
\end{proposition}
\begin{proof}
Let the bundle gerbe $\cG = (P, Y_0)$ where $\pi \colon Y_0 \to X_0$ is a surjective submersion. Let $(\nabla, f)$ be a connection and curving for $\cG$ and denote by $\eta_{(3, 0)} \in \Omega^3(X_0)$ the 
corresponding three-curvature so that $\pi^*(\eta_{(3, 0)}) = df $.  Note that this is unique up to addition to $f$ of $\pi^* (\rho_{(2, 0)})$ 
where $\rho_{(2, 0)} \in \Omega^2(X_0)$, which changes $\eta_{(3, 0)}$ to $\eta_{(3, 0)} + d \rho_{(2, 0)}$. 

 Notice first that $T$ is a trivialisation of $\delta(\cG)$ over $X_1$ and that $\delta(\cG)$ has connection $\delta(\nabla)$ and 
 curving $\delta(f)$. As in the discussion following Lemma \ref{lemma:strong-triv} we choose a connection $\nabla_T$  for $T$ satisfying $\delta_Y( \nabla_T) = \delta(\nabla)$  and define $\eta_{(2, 1)} = -\nu_T \in \Omega^2(X_1)$ so that 
 \begin{equation}
 \label{eq:eta21}
 \pi^*(\eta_{(2, 1)}) = -F_T + \delta(f).
 \end{equation}
   Hence $\pi^*(\delta (\eta_{(3, 0)})) = \delta ( d f)  = \pi^*(d\eta_{(2,1)})$ so that 
$ -d \eta_{(2,1)} + \delta (\eta_{(3, 0)})  = 0$ as required. 

The choice of $\nabla_T$ is unique up to adding $ \pi^*(\rho_{(1, 1)})$ where $\rho_{(1, 1)} \in \Omega^1(X_1)$ which  changes $\eta_{(2, 1)} $ to $\eta_{(2, 1)} - d \rho_{(1, 1)}$.  If we also  change $f$ as above then we change $\eta_{(2, 1)} $ to $\eta_{(2, 1)} - d \rho_{(1, 1)} + \delta (\rho_{(2, 0)})$. 

Notice that  $\delta(T) $ has a connection 
$\delta(\nabla_T) $ whose curvature is $\delta(\nu_T)  = \delta(\eta_{(2, 1)})$.  This descends to a connection $\nabla_{A_T} $ on  $A_T$  which has a trivialising section $s$. We define
\begin{equation}
\label{eq:eta12}
\eta_{(1, 2)} = s^*(\nabla_{A_T}) \in \Omega^1(X_2)
\end{equation}
so that 
$$
\pi^*(\eta_{(1, 2)}) = s^*(\d(\nabla_T)) \in \Omega^1(X_2).
$$

 Moreover $d \pi^*(\eta_{(1, 2)})  = -\pi^* (\delta \eta_{(2, 1)})$.
Hence $ d\eta_{(1, 2)} + \delta (\eta_{(2, 1)})= 0$. 

Notice that if we change $\nabla_T$ by adding $\pi^*(\rho_{(1, 1)})$ then $\eta_{(1, 2)}$ changes by addition of $\d (\rho_{(1, 1)})$. 

Lastly because $\delta(s) = 1$ we conclude that  $\delta (\eta_{(1, 2)} )= 0$. 

Notice that if we change $(\nabla, f)$ to $(\nabla + \delta(\alpha), f + d\alpha)$ for $\alpha \in \Omega^1(Y_0)$, then $\nabla_T$ changes to $\nabla_T + \alpha$ and the cocycle
is unchanged. 

Finally we conclude that the simplicial extension defines  a  cocycle
$$
\eta = (0, \eta_{(1, 2)} , \eta_{(2, 1)}, \eta_{(3, 0)} ) \in  \cA^3(X_\bullet),
$$
 whose class in  $H^3(X_\bullet, \RR)$ we have seen is independent of choices. 
\end{proof}

We note that 

\begin{proposition}
The pullback of simplicial extensions defined in Proposition \ref{prop:simp-ext-pullback} preserves 
simplicial classes.
\end{proposition}
\begin{proof}
This follows immediately from the construction as all the data used to define the class pulls back.
\end{proof}

\begin{definition}
We call the class defined in Proposition \ref{prop:simplicial-class} the (real) {\em extension class} of the simplicial extension and 
denote it by $\ec(Y_\bullet, T, s)$, or simply $\ec(T, s)$, since the class is independent of $Y_\bullet$.
\end{definition}

\begin{remark} 
Although we shall not need to make use of them, it is worth pointing out the 
following facts.   To every simplicial extension of $\cG$ over $X_\bullet$ 
there is associated an {\em integral} extension class in $H^3(\norm{X_\bullet},\ZZ)$ 
which classifies simplicial extensions in the sense that there is an isomorphism 
between a suitable set of equivalence class of simplicial extensions and 
$H^3(\norm{X_{\bullet}},\ZZ)$.  Furthermore, a bundle gerbe $\cG$ over $X_0$ has a simplicial extension 
if and only if there is a bundle gerbe $\widetilde{\cG}$ over the geometric 
realization $\norm{X_{\bullet}}$ whose Dixmier--Douady class in $H^3(\norm{X_{\bullet}},\ZZ)$ 
is the integral extension class of the simplicial extension and whose 
restriction to $X_0$ is stably isomorphic to $\cG$.  A proof of this fact uses 
some of the machinery of simplicial \v{C}ech cohomology (see for example \cite{Freidlander}), which 
would require a lengthy discussion.  Since these facts are not central to our paper 
we have chosen to omit them.    
\end{remark}

\begin{proposition} 
If $(T_i, s_i)$ is a simplicial extension of  $\cG_i$ for $i = 1, 2$ then 
$(T_1 , s_1) \otimes (T_2, s_2) = (T_1 \otimes T_2,  s_1 \otimes s_2)$ is a simplicial extension for $\cG_1 \otimes \cG_2$
 and $\ec(( T_1, s_1)\otimes (T_2, s_2)) = \ec(T_1, s_1) + \ec(T_2, s_2)$. 
 \end{proposition}

When $\cG$ is trivial so that $\cG = (\delta_Y (R), Y)$ for $R \to Y $ we can construct a {\em trivial simplicial extension} $(\delta(R), 1)$, where $\d(R)$ denotes the induced trivialisation of $\d(\cG)$.

If we regard $X_0 \to X_0$ as a surjective submersion we can identify $X_0^{[2]}$ with $X_0$ and the 
trivial line bundle $X_0 \times U(1)$ gives us a strongly trivial bundle gerbe. The product of any bundle
gerbe $\cG$ with this bundle gerbe is naturally isomorphic to itself.  Any simplicial line bundle $(J, \sigma)$ gives us 
a simplicial extension $(J,  \sigma)$ of  this trivial bundle gerbe. 

It follows that a simplicial line bundle can form a product with a  simplicial extension to give rise to a new simplicial 
extension. Or more directly, given a simplicial extension $(Y_\bullet, T,  s)$ with $\pi \colon Y_\bullet \to X_\bullet$ and a simplicial line bundle $(J, \sigma)$, we can 
define a new trivialisation $T \otimes \pi^{-1}(J)$.  Then $A_{T \otimes \pi^{-1}(J)} = A_T \otimes \delta(J)$ which has a section 
$s \otimes \delta(\sigma)$. 
 Hence we have a new simplicial extension  $(T, s) \otimes (J, \sigma) = ( T \otimes \pi^{-1}(J), s \otimes \delta(\sigma))$. 

If we pick a connection $\nabla_J$ for $J$ then we obtain an integral two-form $F_J \in \Omega^2(X_1)$ with $d F_J = 0$. 
Also we can define $\alpha \in \Omega^1(X_2)$ by $\alpha = \sigma^*(\delta(\nabla_J))$ and $\delta(\alpha) = (\delta \sigma)^*(\delta^2(\nabla_J)) = 0$ and 
$\delta(F_J) = d \alpha$. Hence a simplicial line bundle has a simplicial Chern class $\scc(J, \sigma) \in H^3(X_\bullet, \RR)$ 
represented by $(0, \alpha, F, 0)$. Clearly this is in the kernel of the map $H^3(X_\bullet, \RR) \to H^3(X_0, \RR)$.

If $(T, s)$ and $\cG$ has the class
$$
(0, \eta_{(1, 2)} , \eta_{(2, 1)}, \eta_{(3, 0)} )$$
then  $(T \otimes \pi^{-1}(J),  s \otimes \delta(\sigma))$ 
has the class 
$$
(0, \eta_{(1, 2)}+\a , \eta_{(2, 1)} +F, \eta_{(3, 0)} ).
 $$
 and hence $\ec((T, s) \otimes (J, \sigma)) = \ec(T, s) + \scc(J, \sigma)$.
 In fact the converse of this is true. 
 
 \begin{proposition}
Let $(T_1, s_1)$ and $(T_2, s_2)$ be two simplicial extensions of $\cG$. Then there exists a simplicial line
bundle $(T_1 \oslash T_2, s_1 \otimes s_2^*)$ such that $(T_1, s_2) = (T_2, s_2) \otimes (T_1 \oslash T_2, s_1 \otimes s_2^*)$.
\end{proposition}
\begin{proof} 
The construction is straightforward. We have two trivialisations of $\delta(\cG)$, which differ by a line 
bundle $T_1 \oslash T_2 \to X_1$.  This gives us $A_{T_1} = A_{T_2} \otimes \delta (T_1 \oslash T_2)$
so that $\delta (T_1 \oslash T_2)$ has a section $s_1 \otimes s_2^*$. As $\delta(s_1) = \delta(s_2) = 1$ it follows that $\delta(s_1 \otimes s_2^*) = 1$.
\end{proof}

\begin{proposition} 
Let $(T, s)$ be a simplicial extension of $\cG$ and $\rho \colon \cH \to \cG$ be a stable isomorphism.
Then there is a canonically defined simplicial extension $\rho^{-1}(T, s)$ of $\cH$.  Moreover 
$\ec( \rho^{-1}(T, s)) = \ec(T, s)$. 
\end{proposition}

\subsection{Descent for bundle gerbes}\label{SS:descent}

Let $\pi \colon M \to N$ be a surjective submersion. 
  
\begin{proposition} 
\label{prop:easy-descent}
If $\pi \colon M \to N $ is a surjective submersion and $\cG$ is a bundle 
gerbe on $N$ then $\pi^*(\cG)$ admits a simplicial extension to $M^{[\bullet +1]}$. 
\end{proposition}
\begin{proof} 
The projection $M^{[k]} \to N = N^{(k)}$ defines a morphism
of simplicial manifolds $M^{[\bullet + 1]} \to N^{(\bullet)}$ and it suffices to pull back the simplicial extension defined
in Proposition \ref{prop:Psimplicial}.
\end{proof}

The converse of Proposition \ref{prop:easy-descent} is in fact true. Before proving this, we make two definitions as follows:

\begin{definition}
If $M \to N$ is a surjective submersion and $\cG$ a bundle gerbe on $M$ then 
{\em $(M \to N)$-descent data} for $\cG$ is a simplicial extension of  $\cG$ 
over $M^{[\bullet+1]}$. 
\end{definition}

\begin{definition}
If $\pi \colon M \to N$ is a surjective submersion and $\cG$ a bundle gerbe on $M$ then we say that {\em 
$\cG$ descends to $M$} if 
there exists a bundle gerbe $\cH $ over $ N$ with $\pi^*(\cH)$ stably isomorphic to $\cG$.
\end{definition}

\begin{proposition} 
\label{prop:descent}
If $\pi \colon M \to N$ is a surjective submersion and $\cG$ a bundle gerbe on $M$ then $\cG$
descends to $N$ if and only if there exists $(M \to N)$-descent data for $\cG$. 
\end{proposition}
\begin{proof}
We have established  one direction already in Proposition \ref{prop:easy-descent}.
 Let  $\pi \colon M \to N$ be a surjective submersion and $\cG = (P, Y)$  a bundle gerbe over $M$ with a simplicial extension $(T, s)$ for the simplicial manifold $M^{[\bullet+1]}$.  Then let $\overline Y = Y \to N$, regarded as a surjective submersion over $N$. 
Then there is an equality of simplicial surjective submersions $\overline Y^{[\bullet+1]} = \mu^{-1}(Y^{\bullet+1})$. 
The simplicial extension defines a simplicial line bundle over $\mu^{-1}(Y^{\bullet+1})$ and hence  a simplicial line bundle over $\overline Y^{[{\bullet+1}]}$,
i.e.~a bundle gerbe. We take the descended bundle gerbe $\cH$ to be the one defined by this simplicial line bundle.

We need to check that the pullback of $\cH$ is stably isomorphic to $\cG$. Consider $\pi^{-1}(\overline Y) \to M$.  This contains $Y$ so we have a morphism of simplicial manifolds 
$$Y^{[\bullet+1]} \to 
\pi^{-1}(\overline Y^{[\bullet+1]})  \to  \overline Y^{[\bullet+1]} = \mu^{-1}(Y^{\bullet+1}).
$$
 Notice that this composition maps $Y^{[\bullet+1]}$ to the 
subset $Y^{[\bullet+1]} \subset \mu^{-1}(Y^{\bullet+1})$ and is the identity.   It follows that  if we start with the 
simplicial extension $(T, s)$ as a simplicial line bundle on  $ \mu^{-1}(Y^{\bullet+1}) $, we descend by regarding it as a simplicial line bundle on $\overline Y^{[\bullet+1]}$ and we pullback to $\pi^{-1}(\overline Y^{[\bullet+1]})$ and restrict 
to $Y^{[\bullet+1]}$, that is the same as just restricting $(T, s)$ to $Y^{[\bullet+1]}$, which by Proposition \ref{prop:Prestriction}
we know to be $(P, Y)$. Hence $\pi^{-1}(\cH)$ is stably isomorphic to $\cG$.
\end{proof}

We are interested in several particular cases of simplicial extensions, arising from actions of Lie groups and actions of 2-groups. We devote the rest of this paper to the study of these.

\section{Equivariant bundle gerbes}\label{S:eq}

\subsection{Strong and weak group actions on bundle gerbes}\label{SS:group actions}

Recall that if  $M $ is a  manifold on which a Lie group $G$ acts smoothly on the right we have the simplicial manifold  $EG(M)_\bullet $ from Example \ref{ex:EG(M)}.  We define:

\begin{definition}
If $G$ acts smoothly on $M$ and $\cG$ is a bundle gerbe on $M$ a {\em weak action} of $G$ on $\cG$  is a  simplicial 
extension   for $\cG$ over $EG(M)_\bullet$. 
\end{definition}

Notice that the simplicial extension class of a weak action lives in $H^3(EG(M)_\bullet, \RR) = H^3_G(M, \RR)$, the equivariant
de Rham cohomology of $M$. 

\begin{definition}[c.f. \cite{Gom2, MatSte, Mei}]
Let $\cG = (P, Y)$ be a bundle gerbe over $M$.  A {\em strong action} of $G$ on $\cG$ is a smooth 
action of $G$ on $Y$ covering the action on $M$ and a smooth action of $G$ on $P \to Y^{[2]}$ by bundle morphisms covering the induced action 
on $Y^{[2]}$ and commuting with the bundle gerbe product.  We say that $\cG$ is a {\em strongly equivariant bundle gerbe}.
\end{definition}

\begin{proposition} 
\label{prop:strong2weak}
A strong action of $G$ on $\cG$ induces a weak action. 
\end{proposition}
\begin{proof} 
We construct a simplicial extension $(EG(Y)_{\bullet}, T, s)$ of $\cG = (P, Y)$ over $EG(M)_\bullet$. First we need a trivialisation of $(\d(P), Y\times G)$, which has fibre over $(y_1, y_2, g) \in Y^{[2]} \times G = (Y \times G) \times_{M \times G} (Y \times G)$ given by $P_{(y_1, y_2)}^* \otimes P_{(y_1 g, y_2 g)}$. Hence, $(\d(P), Y \times G)$ is strongly trivial via the isomorphism $\xymatrix@1{P_{(y_1, y_2)}^* \otimes P_{(y_1 g, y_2 g)} \ar@{|->}[r]^>>>>{1 \otimes g^{-1}} & P_{(y_1, y_2)}^* \otimes P_{(y_1 , y_2 )} \simeq U(1)}.$ We therefore take $T$ to be the trivial bundle $(Y \times G) \times U(1)$. Thus $\d(T)$ is trivial and descends to the trivial bundle with its canonical section $s$, and so $\d(s) = 1$.
\end{proof}

\subsection{Descent for equivariant bundle gerbes}

If $\pi \colon M \to N$ is a principal $G$-bundle it is straightforward to show  that if  a bundle gerbe
$\cG = (P, Y)$ on $M$ admits a strong action of $G$ then it descends to a quotient bundle gerbe  $(P/G, Y/G)$ on $N$. We now show that even for a weak action of $G$ bundle gerbes descend.

\begin{proposition}
Let $M \to N$ be a principal $G$-bundle.  If $\cG$ is a bundle gerbe on $M$ acted on weakly by $G$ then $\cG$ descends to a bundle gerbe on $N$, which is given explicitly by Proposition \ref{prop:descent}.
\end{proposition}
\begin{proof}
The proof is straightforward and only requires us to show that a weak $G$ action on $\cG$ is the same 
as $(M \to N)$-descent data for $\cG$.  Equivalently, we need to show that the
simplicial manifolds $EG(M)_\bullet$ and $M^{[\bullet + 1]}$ are isomorphic, which is true by Lemma \ref{L:EG(M)=M^[bullet]}.
\end{proof}

A similar result is proved in  \cite{Gom} on the level of cohomology using the definition of equivariance from \cite{Bry}.

  In the case of a strong action we now have two ways to descend the bundle gerbe. These are related by

\begin{proposition}\label{P:two descents are the same}
Let  $M \to N$ be a principal $G$-bundle and $\cG = (P, Y)$ a bundle gerbe on $M$ acted on strongly by $G$. The quotient
of $\cG$ by the strong $G$ action and the descent of $\cG$ by the induced weak $G$ action are stably isomorphic.
\end{proposition}
\begin{proof}

We defined the descent of the bundle gerbe $(P, Y)$ in Proposition \ref{prop:descent} given a weak action of $G$ in the form of a simplicial extension $(\mu^{-1}(Y^{\bullet + 1}), T, s)$ of $(P, Y)$ over $M^{[\bullet + 1]}$. This was given by the observation that $\mu^{-1}(Y^{\bullet + 1}) = \overline Y^{[\bullet+1]}$ (for $\overline Y \to M / G$ the submersion given by the composition $Y \to M \to M/G$) and then pulling back the simplicial extension to a simplicial line bundle $[T, s]$ over $\overline Y^{[\bullet + 1]}$. We also have a quotient bundle gerbe $(P/G, Y/G) $ over $M/G $. We  define a map of surjective submersions $\rho \colon \overline Y   \to Y/G $ over $N = M/G$ and show that ${\rho}^{-1}(P/G) \simeq T$ and that this map commutes with the bundle gerbe multiplication on $T \to \overline Y^{[2]}$ and $P/G \to (Y/G)^{[2]}$. 

Notice, however, that the weak action on $(P, Y)$ induced by the strong action of $G$, given in Proposition \ref{prop:strong2weak}, is a simplicial extension $(EG(Y)_\bullet, Y \times U(1), 1)$. Proposition \ref{prop:there exists a mu extension} tells us that this is the pullback of a simplicial extension $(\mu^{-1}(Y^{\bullet + 1}), T, s)$. Therefore, we need the simplicial line bundle coming from this simplicial extension. The trivialisation $T$ is given by Proposition \ref{prop:descent-trivialisations} as follows. We have the map $EG(Y)_1 = Y\times G \to \mu^{-1}(Y^{\bullet + 1}) =  \overline Y^{[2]} ; \ (y, g) \mapsto (y, yg)$. Then for $(y_0, y_1), (y_0'. y_1') \in \overline Y^{[2]}\times_{M \times G}  \overline Y^{[2]} $ we have $\d(P)_{(y_0, y_1), (y_0'. y_1')} = P_{(y_0, y_0')}^* \otimes P_{(y_1, y_1')} $. Note that $\pi(y_1) = \pi(y_0) g$ for some $g \in G$ and $\pi\colon Y \to M$, and similarly for $y_0', y_1'$, and that $\pi(y_i) = \pi(y_i')$. Then Proposition \ref{prop:descent-trivialisations} gives
\begin{align*}
T_{(y_0, y_1)}	&= U(1) \otimes \d(P)_{(y, yg)(y_0, y_1)}\\
			&= U(1) \otimes P_{(y, y_0)}^* \otimes P_{(yg, y_1)}\\
			&= U(1) \otimes P_{(y, y_0)}^* \otimes P_{(y, y_1g^{-1})}\\
			&= U(1) \otimes P_{(y_0, y_1g^{-1})}\\
			&= P_{(y_0, y_1g^{-1})},
\end{align*}
where again $g$ is such that $\pi(y_1) = \pi(y_0) g$. To construct the descended bundle gerbe we need the section of $\d(T) \to \overline Y^{[3]}$. For $(y_0, y_1, y_2) \in \overline Y^{[3]}$ and $\pi(y_i) = \pi(y_j) g_{ij}$ we have
$$
\d(T)_{(y_0, y_1, y_2)} = P_{(y_1, y_2 g_{12}^{-1})} \otimes  P_{(y_0, y_2 g_{02}^{-1})}^* \otimes  P_{(y_0, y_1 g_{01}^{-1})},
$$
and a section is given by $s(y_0, y_1, y_2) = p_{12} \otimes (p_{12}g_{01} \cdot p_{01})^* \otimes p_{01}$, where $p_{01} \in P_{(y_0, y_1)}, p_{12} \in P_{(y_1, y_2)}$ and $\cdot$ denotes the bundle gerbe multiplication in $P$.

We can now define $\rho \colon \overline Y \to Y/G $ by $\rho(y) = yG \in Y/G$ where the latter denotes the orbit of $y$ under $G$. Then we have a map $T_{(y_0, y_1)} = P_{(y_0, y_1g^{-1})} \to (P/G)_{(y_0G, y_1 G)}$ because the $G$ orbit of $(y_0, y_1g^{-1})$ is the pair of $G$ orbits $(y_0 G, y_1 g^{-1} G) = (y_0 G, y_1 G)$.  Hence we have described a bundle map $T \to P/G$ covering the induced map $\rho \colon \overline Y^{[2]} \to (P/G)^{[2]}$.

We need to prove that this map preserves the bundle gerbe product. The multiplication in $(P/G, Y/G)$ is given by the section $\sigma( y_0G, y_1G, y_2G) = p_{12} \otimes (p_{12} \cdot p_{01} G) \otimes p_{01}$ and it is easy to see that $\rho$ maps the section $s$ to $\sigma$ because the $G$ action on $P$ commutes with the bundle gerbe multiplication. It follows that the bundle gerbe product is preserved.
\end{proof}

\subsection{The class of a strongly equivariant bundle gerbe}

Assume that $G$ acts strongly on the bundle gerbe  $\cG = (P, Y)$  over $M$. Choose a bundle gerbe connection $\nabla$ for $P$ and a curving $f$. Let $\omega_{(3, 0)} \in \Omega^3(M)$ be the three-curving. We show how to write down an equivariant three-class for $\cG$.

Over $Y^{[2]} \times G$ there are two bundles $d_0^{-1}(P) $ and $d_1^{-1}(P)$ corresponding  to the bundle gerbes $d_0^{-1}(\cG)$ and $d_1^{-1}(\cG)$ over
$M \times G$. 
Let $\phi \colon d_0^{-1}(P) \to d_1^{-1}(P) $ be 
the action of right multiplication by $g^{-1}$.  On $d_0^{-1}(P)$ there 
are two connections: $d_0^{-1}(\nabla)$ and $ \phi^{-1} d_1^{-1}(\nabla) \phi$.  They 
are both bundle gerbe connections so we must have 
\begin{equation}
\label{eq:beta}
 d_0^{-1}(\nabla) - \phi^{-1} d_1^{-1}(\nabla) \phi = \delta_Y(\beta) ,
\end{equation}
for $\beta \in \Omega^1(Y \times G)$.  Similarly we have curvings $d_0^*(f)$ and $d_1^*(f)$ and 
$$
\delta_Y( d_0^*(f) - d_1^*(f) - d\beta) = 0,
$$
or 
\begin{equation}
\label{eq:omega21}
d_0^*(f) - d_1^*(f) - d\beta = \pi^*(\omega_{(2, 1)}),
\end{equation}
for $\omega_{(2, 1)}  \in \Omega^1(M \times G)$.  Moreover
\begin{align*}
\pi^*(d\omega_{(2, 1)}) &=  d_0^*(f) - d_1^*(f) \\
                      & = d_0^*(df) - d_1^*(df)\\
                       & = d_0^*(\pi^*(\omega_{(3, 0)})) - d_1^*(\pi^*(\omega_{(3, 0)})),
                       \end{align*}
 so that 
 $$
- d\omega_{(2, 1)} + \delta( \omega_{(3, 0)}) = 0.
 $$

Applying $d^{-1}_0$, $d^{-1}_1$ and  $d^{-1}_2$ to  \eqref{eq:beta} we obtain
\begin{align}
\mu_2^{-1}(\nabla) - \phi_0^{-1} \mu_1^{-1}(\nabla) \phi_0 &= \delta_Y(d_0^*(\beta))  \nonumber \\
\mu_2^{-1}(\nabla) - \phi_1^{-1} \mu_0^{-1}(\nabla) \phi_1 &= \delta_Y(d_1^*(\beta))   \nonumber \\
\mu_1^{-1}(\nabla) - \phi_2^{-1} \mu_0^{-1}(\nabla) \phi_2 &= \delta_Y(d_2^*(\beta)) \label{eq:lastline}\end{align}
where here $\mu_0(m, g_1, g_2) = m$, $\mu_1(m, g_1, g_2) = m g_1$ and $\mu_2(m, g_1, g_2) = m g_1 g_2$ and 
$\phi_i = \phi \circ d_i$. We have $\phi_1 = \phi_2 \phi_0$ so conjugating line \eqref{eq:lastline}
by $\phi_0^{-1}$ we obtain 
$$
\phi_0^{-1}\mu_1^{-1}(\nabla)\phi_0 - \phi_1^{-1} \mu_0^{-1}(\nabla) \phi_1 = \delta_Y(d_2^*(\beta)),
$$
and an alternating sum gives us $\delta_Y(\delta(\beta)) = 0$. Hence 
 \begin{equation}
 \label{eq:omega12}
  \delta(\beta) = \pi^*(\omega_{(1, 2)}),
  \end{equation}
   for some $\omega_{(1, 2)} \in \Omega^1(M \times G^2)$.
It then follows that 
$$
\pi^*(\delta(\omega_{(2, 1)})) = - d \delta(\beta) = -\pi^*(d\omega_{(1, 2)}),
$$
or 
$$
d \omega_{(1, 2)} + \delta(\omega_{(2, 1)}) = 0.
$$
Notice also that $0 = \delta^2(\beta) = - \pi^*(\delta(\omega_{(1, 2)}))$ so that 
$\delta(\omega_{(1, 2)}) = 0 = d 0 $.
Thus we have defined a cocycle 
$$
\omega = ( 0, \omega_{(1, 2)}, \omega_{(2, 1)}, \omega_{(3, 0)}) \in \cA^3(EG(M)_\bullet).  
$$

Consider what happens if we vary the choices involved. We could replace $\beta$ by $\beta + \pi^*(\rho_{(1, 1)})$, changing $\omega_{(2, 1)}$ by adding $- d \rho_{(1, 1)}$
and  $\omega_{(1, 2)}$ by adding $\d (\rho_{(1, 1)})$, which leaves the class of $\omega$ unchanged. Also we could replace the curving $f$ by adding $\pi^*(\rho_{(2, 0)})$ to 
it and changing $\omega_{(3, 0)}$ by addition of $d \rho_{(2, 0)}$, and $\omega_{(2, 1)}$ by addition of $\delta (\rho_{(2, 0)} )$, which again leaves the class of $\omega$ unchanged. 
Finally, we can change $(\nabla, f)$ to $(\nabla + \delta(\alpha), f + d\alpha)$ for $\alpha \in \Omega^1(Y)$, which changes $\beta$ to $\beta + \delta(\a)$. The left hand side of equation \eqref{eq:omega21} then changes by the addition of $\delta(d \a) - d \delta(\a) = 0$, leaving  $\omega$  unchanged.  We conclude that the class of $\omega$ depends only on the strong group action and the bundle gerbe.

\begin{definition}
We call the class just defined the {\em strongly equivariant class} of the strongly equivariant bundle gerbe $\cG$ and denote it by $\sec(\cG)$.
\end{definition}

\noindent
In \cite{Mei} Meinrenken defines the  class of a strongly equivariant  bundle gerbe using the Cartan model of equivariant cohomology. See also related work of Stienon \cite{Sti}.

With these observations we can prove the following result:

\begin{proposition}
The equivariant class of a strongly equivariant gerbe is equal to the 
simplicial extension class of the corresponding simplicial extension.
\end{proposition}
\begin{proof}
In Proposition \ref{prop:simplicial-class} we constructed the class $\eta = (0, \eta_{(1, 2)} , \eta_{(2, 1)}, \eta_{(3, 0)} ) $, which we compare to the class $\omega = ( 0, \omega_{(1, 2)}, \omega_{(2, 1)}, \omega_{(3, 0)})$ above. 

Firstly, it is clear that $\omega_{(3, 0)} = \eta_{(3, 0)}$. 

Recall from Proposition \ref{prop:strong2weak} that the simplicial extension corresponding to the strong action of $G$ on $\cG$ is given by $(EG(Y)_\bullet, T, s)$, where $T$ (and hence $A_T$) is the trivial bundle and $s$ is the canonical section of $A_T$. Equation (\ref{eq:beta}) tells us we can choose the trivialising connection on $T$ to be $\beta$ and then equation (\ref{eq:omega21}) is the same as equation (\ref{eq:eta21}) and hence $\eta_{(2, 1)} = \omega_{(2, 1)}$. 

Finally, the induced connection on $\d(T)$ is given by $\d(\b)$ and hence comparing equation (\ref{eq:omega12}) with equation (\ref{eq:eta12}) implies that $\eta_{(1, 2)} = \omega_{(1, 2)}$. 
\end{proof}

\section{The basic bundle gerbe}\label{S:basic}

We review the constructions in \cite{MurSte} and situate them in the equivariant setting. 
We first recall from \cite{MurSte} the {\em basic bundle gerbe} on $G= U(n)$ and the canonical connection and curving on it constructed using the holomorphic functional calculus.  

Write $Z=S^1\setminus \{1\}$. Define 
$Y\subset Z\times G$ to be the set of pairs $(z,g)$, where $z$ is not an eigenvalue of $g$.  
We equip $Z$ with an ordering via the identification of $Z$ with the open interval 
$(0,2\pi)$ by $\phi\mapsto \exp(i\phi)$.  Let $\pi\colon Y\to G$ denote the canonical map.  We note
that elements of $Y^{[2]}$ can be identified with triples $(z_1, z_2, g)$ where $(z_1, g), (z_2, g) \in Y$. 
In such a case if  $z \in Z$ we say that it is  {\em between} $z_1$ and $z_2$ if it is in the component of $S^1 \setminus \{z_1, z_2\}$
not containing $\{1\}$. 

As described in \cite{MurSte}, there is a natural line bundle $L$ on $Y^{[2]}$ 
together with a bundle gerbe product on $L$, giving $(L,Y)$ the structure of a 
bundle gerbe on $G$.  To describe this note first that there is a decomposition of $Y^{[2]}$ as a union of three disjoint open sets defined by:
 \begin{align*}
Y_+^{[2]} & = \{  (z_1, z_2, g) \mid  z_1 < z_2 \text{ and there is some eigenvalue of $g$ between $z_1$ and $z_2$}\}\\
Y_-^{[2]} & = \{  (z_1, z_2, g) \mid  z_1 > z_2 \text{ and there is some eigenvalue of $g$ between $z_1$ and $z_2$}\}\\
\intertext{and}
Y^{[2]}_0  &= \{ (z_1, z_2, g) \mid \text{there is no eigenvalue of $g$ between $z_1$ and $z_2$} \}.
\end{align*}

If $(z_1,z_2,g)\in Y_+^{[2]}$ we define 
\[
E_{(z_1,z_2,g)} = \bigoplus_{z_1<\lambda<z_2}E_{\lambda}(g), 
\]
where $E_\lambda(g)$ denotes the $\lambda$-eigenspace of $g$ and we write 
$z_1<\lambda<z_2$ to indicate that $\lambda$ is between $z_1$ and $z_2$.  It is shown in \cite{MurSte} that $E\to Y_+^{[2]}$ 
is a smooth, locally trivial vector bundle.  Recall also from \cite{MurSte} that the 
orthogonal projection $P\colon Y_+^{[2]}\to M_n(\CC)$ onto $E$ is given by the 
contour integral formula 
\begin{equation} 
\label{eq:contour integral formula for P}
P(z_1,z_2,g) = \frac{1}{2\pi i}\oint_{C_{(z_1,z_2,g)}}(\xi 1 -g)^{-1}d\xi 
\end{equation}
where $C_{(z_1,z_2,g)}$ is an anti-clockwise curve enclosing all of the eigenvalues 
of $g$ between $z_1$ and $z_2$.

The line bundle $L\to Y^{[2]}$ is defined as follows.  If $(z_1,z_2,g)\in Y_+^{[2]}$ 
we set 
\[
L_{(z_1,z_2,g)} = \bigwedge^{\mathrm{top}} E_{(z_1,z_2,g)} .
\]
If $(z_1,z_2,g)\in Y_-^{[2]}$ we set $L_{(z_1,z_2,g)} = L_{(z_2,z_1,g)}^*$.  
If $(z,z,g)\in Y_0^{[2]}$ we set $L_{(z,z,g)} = \CC$.  It is proven in \cite[Proposition 3.1]{MurSte}  that $L\to Y^{[2]}$ is a smooth, locally trivial, hermitian line bundle.  
Furthermore it is shown in \cite{MurSte} that there is a natural bundle gerbe 
product on $L$, equipping $(L,Y)$ with the structure of a bundle gerbe.  The resulting 
bundle gerbe $\cB_n = (L,Y)$ is a model for the basic bundle gerbe on $G = U(n)$.      

Observe that $G$ acts smoothly on $Y$ from the right, covering the adjoint action of $G$ 
on itself.  More precisely, we define $Y\times G\to Y$ by $((z, g),h)\mapsto 
(z, h^{-1}gh)$; note that the projection map $\pi\colon Y\to G$ is equivariant.  
We have the following lemma.  
 
\begin{lemma} 
\label{lem:basic bg is strongly equiv}
The basic bundle gerbe $\cB_n = (L,Y)$ on $G = U(n)$ is a strongly equivariant bundle 
gerbe for the adjoint action of $G$ on itself.  
\end{lemma} 

\begin{proof} 
We need to show that the induced action $ Y^{[2]}\times G\to Y^{[2]}$ lifts to 
an action of $G$ on $L$.  It is sufficient to prove that the action of $G$ on $Y_+^{[2]}$ 
lifts to an action of $G$ on $E$; this follows from the fact that the left action 
of $G$ on $\CC^n$ is smooth and the fact that  
if $v$ is an eigenvector of $g$ with eigenvalue $\lambda$, 
then $vh$ is an eigenvector of $h^{-1}gh$ with eigenvalue $\lambda$.      
\end{proof} 

The map $P\colon Y_+^{[2]}\to M_n(\CC)$ extends in an obvious way to 
a smooth map $P\colon Y^{[2]}\to M_n(\CC)$.  Observe that $P$ satisfies 
\begin{equation}
\label{eq:equiv for P}
d_1^*P = \Ad_{p_2}d_0^*P 
\end{equation}
on $ Y^{[2]}\times G$, 
where $p_2\colon  Y^{[2]} \times G\to G$ is the map $p_2((z_1,z_2, g),h) = h$.

Recall from \cite{MurSte} that there is a canonical bundle gerbe connection $\nabla$ and curving $f$ on 
$(L,Y)$ whose 3-curvature is equal to the basic 3-form 
\[
\nu = -\frac{1}{24\pi^2} \tr(g^{-1}dg)^3.  
\]
We briefly review the construction of $\nabla$ and $f$ as they will be needed in the 
sequel.  The orthogonal projection $P\colon Y_+^{[2]}\to M_n(\CC)$ induces 
a connection $\nabla_E$ on $E$ by projecting the trivial connection $d$ on 
$Y_+^{[2]}\times \CC^n$ to $E$.  The connection $\nabla_E$ then induces 
a connection $\nabla$ on the restriction of $L$ to $Y_+^{[2]}$, over $Y_-^{[2]}$ 
we equip $L$ with the dual connection and over $Y_0^{[2]}$ we take the 
flat connection.   It is proven in \cite{MurSte} that this connection $\nabla$ 
on $L$ is a bundle gerbe connection and that moreover a curving 
$f$ for $\nabla$ is given by the 2-form on $Y$ defined by 
\[
f(g,z) = \frac{1}{8\pi^2}\oint_{C_{(g,z)}}\log_z \xi \tr((\xi 1 -g)^{-1}dg(\xi 1 - g)^{-2}dg)d\xi 
\]
where $C_{(g,z)}$ is an anti-clockwise contour in $\CC\setminus R_z$ enclosing the 
eigenvalues of $g$, and where $R_z$ denotes the closed ray from the origin in 
$\CC$ through $z$.  Here $\log_z\colon \CC\setminus R_z\to \CC$ is the branch 
of the logarithm defined by making the cut along $R_z$ and defining $\log_z(1) = 0$.     

The connection $\nabla$ on $L$ is not equivariant however, for the action 
of $G$ on $L$ described in Lemma~\ref{lem:basic bg is strongly equiv} above.  
We investigate the failure of $\nabla$ to be equivariant more closely.  We have 
an isomorphism of line bundles $\phi\colon d_0^*L\to d_1^*L$ over $ Y^{[2]}\times G$; 
if $s$ is a section of $d_0^*L$ over $ Y^{[2]} \times G$ then $\phi(s)$ is the section 
of $d_1^*L$ over $Y^{[2]} \times G$ defined by 
\[
\phi(s) = s\cdot p_2, 
\]
where $(s\cdot p_2)((z_1,z_2, g), h) = s(z_1,z_2, g)\cdot h$.  We then 
have 
\begin{align*} 
\phi^{-1}d_0^*\nabla (\phi(s)) & = \phi^{-1}\det(d_0^*P(d(s\cdot p_2))) \\ 
& = \phi^{-1}\det(d_0^*P(ds\cdot p_2 + s\cdot dp_2)) \\ 
& = \phi^{-1}\det(\Ad_{p_2}d_0^*P(ds\cdot p_2 + s\cdot dp_2)) \\ 
& = \tr(p_2^*\theta d_1^*P) \cdot s + d_1^*\nabla s ,
\end{align*} 
where we have used~\eqref{eq:equiv for P} and where we have written $\theta$ 
for the right Maurer-Cartan 1-form on $G$.  Using~\eqref{eq:contour integral formula for P} 
we may express the 1-form $\alpha = \tr(p_2^*\theta d_1^*P)$ as a contour integral: 
\[
\alpha(z_1,z_2, g, h) = \frac{1}{2\pi i}\oint_{C_{(g,z_1,z_2)}}\tr(\theta(h)(\xi 1 - g)^{-1})d\xi, 
\]
where, as in~\eqref{eq:contour integral formula for P} above, $C_{(g,z_1,z_2)}$ denotes 
a contour enclosing the eigenvalues of $g$ between $z_1$ and $z_2$, oriented counter-clockwise.  

Since $d_0^*\nabla$ and $d_1^*\nabla$ are bundle gerbe connections, it follows that 
$\delta(\alpha) = 0$ and hence $\alpha = \delta(\beta)$ for some 1-form $\beta$ on 
$Y\times G$.  Using an identical argument to that used in the proof of 
part (a) of Theorem 5.1 in \cite{MurSte}, we obtain the following expression for 
$\beta$: 
\begin{equation} 
\label{eq:eqn for beta}
\beta(z,g,h) = -\frac{1}{4\pi^2}\oint_{C_{(g,z)}}\log_z \xi \tr(\theta(h)(\xi 1 - g)^{-1})d\xi 
\end{equation}
where $\log_z$ and $C_{(g,z)}$ are respectively the branch of the logarithm and the contour described above.  

The main result of this section is the following theorem.  

\begin{theorem} \label{thm:equivariant gerbe on U(n)}
Let $G = U(n)$.  Then the strongly equivariant bundle gerbe $(L,Y)$ has an 
equivariant bundle gerbe connection and curving given by $(\nabla, f,\beta)$, where 
$\nabla$ is the bundle gerbe connection on $L$ described above, $f$ is the 
curving 2-form on $Y$ given by 
\[
f(z, g) = \frac{1}{8\pi^2}\oint_{C_{(g,z)}}\log_z \xi \tr((\xi 1 -g)^{-1}dg(\xi 1 - g)^{-2}dg)d\xi 
\]
and $\beta$ is the 1-form on $Y\times G$ given by 
\[
\beta(z,g,h) = -\frac{1}{4\pi^2}\oint_{C_{(g,z)}}\log_z \xi \tr(\theta(h)(\xi 1 - g)^{-1})d\xi 
\]
where the contour $C_{(g,z)}$ and the branch of the logarithm are described as above.  
Furthermore, the equivariant 3-curvature of this connection and curving is the cocycle 
$(\nu,\omega, 0, 0)\in \mathcal{A}^3(EG(G)_\bullet)$ where 
\begin{align*} 
\nu  & = -\frac{1}{24\pi^2} \tr(g^{-1}dg)^3 \in \Omega^3(G)\\ 
\omega & = \frac{i}{4\pi} \left(\tr(\hat{\theta}_h\theta_h) + 
\tr(\theta \theta_h) + \tr(\theta \hat{\theta}_h)\right) \in \Omega^2(G^2) 
\end{align*}
where we have defined $\theta = g^{-1}dg$, $\theta_h = dhh^{-1}$ and 
$\hat{\theta}_h = g^{-1}\theta_hg$.  Hence the strongly equivariant class of $\cB_n$ is $\sec(\cB_n) = [\nu,\omega, 0, 0]$.

\end{theorem}   

\begin{proof} 
We need to show that the following equations hold: 
\begin{align*} 
& \delta(\nabla) = \delta(\beta)\\
& \delta(f) - d\beta = \pi^*\omega \\    
& \delta(\beta) = 0 .
\end{align*} 
We have established the first equation in the paragraphs preceding the 
statement of the theorem.  We show that the third equation is satisfied, i.e.\ 
that $\delta(\beta) =0$.  We have $d_0^*\beta((z,g),h,k) = \beta((z,h^{-1}gh),k)$, $d_1^*\beta((z, g),h,k) = 
\beta((z,g),hk)$ and $d_2^*\beta((z,g),h,k) = \beta((z,g),h)$.  Therefore we have that 
$\delta(\beta)(z, g, h, k)$ is equal to 
\[
 \frac{1}{4\pi^2}\oint_{C_{(g,z)}} 
\log_z\xi \tr\left[\theta(k)h^{-1}(\xi1 -g)^{-1}h - \theta(kh)(\xi 1 -g)^{-1} + \theta(h)(\xi 1 -g)^{-1}\right] d\xi ,\\ 
\]
which is easily seen to equal 0  
using $\theta(hk) = h\theta(k)h^{-1} + \theta(h)$.  The proof that the second equation is 
satisfied is long and technical and we have therefore relegated it to Appendix \ref{app:long calc for theorem 5.2}.   
\end{proof}

As an illustration of this theorem we consider the case where $G = U(1)$ in detail.  In 
this case the bundle gerbe on $G$ is necessarily trivial.  However, the equivariant bundle 
gerbe on $G$ is non-trivial.  The theorem above shows that its equivariant 3-curvature 
is given by $(0,\omega_{(2,1)})$, where $\omega_{(2,1)}$ is the closed 2-form on $U(1)\times U(1)$ given by 
\[
\omega_{(2, 1)}(\phi_1, \phi_2) = \frac{1}{4 \pi^2} \  d\phi_1 \wedge d\phi_2,
\]
where we have defined $g_1 = \exp(i\phi_1)$ and $g_2 = \exp(i \phi_2)$.
An easy calculation shows that the class in $H^2(U(1)\times U(1), \RR)$ represented 
by $\omega_{(2,1)}$ is non-zero.  It follows that the class in $H^3_{U(1)}(U(1), \RR)$ represented 
by $(0,\omega_{(2,1)})$ is non-zero.  

Note that there also is a non-trivial \emph{multiplicative} bundle gerbe on $U(1)$, with trivial underlying bundle gerbe, using the line bundle on $U(1) \times U(1)$ with curvature $\omega_{(2, 1)}$ \cite{Ganter}.

\section{String structures and simplicial extensions}\label{S:string}
 
Waldorf \cite{Wal} has described string structures on a principal $G$-bundle $P \to X$ as trivialisations of a certain bundle 2-gerbe, called the \emph{Chern--Simons bundle 2-gerbe of $P$}. In this section we show that such a trivialisation gives rise to a simplicial extension of a bundle gerbe. Unlike the examples so far, this is an example of a simplicial extension over a simplicial manifold that is \emph{not} the nerve of a Lie groupoid.

\subsection{Crossed modules}\label{SS:xm}

We shall begin by recording some relevant facts about crossed modules, which will be important in what follows. Crossed modules were introduced by Whitehead in the 1940's as a model for homotopy 2-types. We first recall the definition.

\begin{definition}\label{D:crossed module}
A crossed module $\cK$ is a pair of groups $\hat{K} $ and $L$ together with a homomorphism $\hat{K} \xrightarrow{t} L$ and an action $L \times \hat{K} \xrightarrow{\a} \hat{K}$ by group automorphisms satisfying 
\begin{enumerate}
\item $t(\a(l, k)) = \Ad_l (t(k))$,
\item $\a(t(k_1), k_2) = \Ad_{k_1}(k_2)$,
\end{enumerate}
for $l \in L$ and $k, k_1, k_2 \in \hat{K}$. We shall further assume that $\hat K \to K$ is a locally trivial principal $\ker t$-bundle, where here $K := t(\hat{K})$.
\end{definition}

\begin{remark}
Although we do not need this point of view we remark that a crossed module gives rise to a groupoid $\xymatrix@1@C=3ex{\hat K \times L\, \ar@<0.6ex>[r] \ar@<-0.6ex>[r] & \, L}$ where both the objects and arrows are groups and the source and target are homomorphisms. Further, there are functors $1 \to \cK$ and $\cK \times \cK \to \cK$ making certain natural diagrams commute. Such a thing is the same as a group object in the category of groupoids, and is called a \emph{strict 2-group}. A detailed discussion of this would take us too far afield, so we instead refer the interested reader to history, discussion and definitions in \cite{BaezLauda}
\end{remark}

We also want to say what it means for a crossed module to act on a manifold.

\begin{definition}\label{D:crossed module action}
A \emph{strict action} of a crossed module $ \hat{K} \xrightarrow{t} L $ on a manifold $P$ is an ordinary group action of $L$ on $P$ such that the action restricted to $K < L$ is trivial. 
\end{definition}

\begin{remark}
Although Definition \ref{D:crossed module action} will suffice for our purposes, we remark that there is a definition of a strict action of a strict 2-group $\cK$ on a manifold $P$ given in terms of a functor $\cK \times P \to P$ (where $P$ is considered as a groupoid with no non-identity arrows) making certain diagrams commute. In the case that the 2-group $\cK$ comes from a crossed module $\hat K \xrightarrow{t} L$ it is easy to see that this is equivalent to Definition \ref{D:crossed module action}.
\end{remark}

\begin{example}\label{ex:string}
The crossed module in which we are interested is the following \cite{BSCS}: $\hat K$ is the central extension of the loop group $\widehat{\Omega G}$, and $L$ is the path group $PG$. The map $\widehat{\Omega G} \xrightarrow{t} PG$ is the composition of the projection to $\Omega G$ with the inclusion $\Omega G \hookrightarrow PG$ (so $K = t(\hat K) = \Omega G$ and $\ker t = U(1)$) and the map $\a \colon PG \times \widehat{\Omega G} \to \widehat{\Omega G}$ is a lift of the adjoint map $\Ad \colon PG \times \Omega G \to {\Omega G}$, which we also denote by $\Ad$. The result of \cite{BSCS} is that this defines a crossed module that gives a 2-group model for the 3-connected cover of $G$, the String group of $G$. 

Let $N$ be a manifold with a $G$-action. The crossed module in the previous paragraph acts on $N$ in a natural way via the evaluation map $PG \to G$.

This crossed module will be important in what follows since the simplicial manifold we consider in Section \ref{SS:simplicial extension} is built from the crossed module action on the total space of a $G$-bundle.
\end{example}

We have the following facts about $\hat K \xrightarrow{t} L$:
\begin{enumerate}
\item Since $\hat K$ is a central extension of $K$ it is multiplicative as a principal bundle, that is the following diagram is a pullback
\begin{equation}\label{E:mult}
\begin{aligned}
\xymatrix{
\hat K \hat\otimes \hat K \ar[r] \ar[d]	& \hat K\ar[d]\\
K \times K \ar[r]		& K
}
\end{aligned}
\end{equation}
where $\hat K \hat \otimes \hat K$ denotes the external tensor product, $\hat K_1 \otimes \hat K_2$, where $\hat K_i$ is the pullback of $\hat K$ by the projection onto the $i^{\text{th}}$ factor.\\
\item Since $\hat K \xrightarrow{t} L$ is a crossed module, the map $\a$ lifts the restriction to $K$ of the adjoint map, $\left.\Ad \right|_{K} \colon L \times K \to K$, so the following diagram is a pullback
\begin{equation}\label{E:Ad}
\begin{aligned}
\xymatrix{
L \times \hat K \ar^-{\a}[r] \ar_{\id \times t}[d]	& \hat K\ar[d]\\
L \times K \ar^-{\Ad}[r]		& K
}
\end{aligned}
\end{equation}\\
\item The natural map from the dual bundle $\hat K^*$ to $\hat K$ covers the inverse map on $K$ so that the following diagram is a pullback
\begin{equation}\label{E:inv}
\begin{aligned}
\xymatrix{
\hat K^* \ar[r] \ar[d]			&\hat  K\ar[d]\\
K \ar^-{\phantom{{}^{-1}}(\cdot)^{-1}}[r]	& K
}
\end{aligned}
\end{equation}
\end{enumerate}
In terms of the fibres of $\hat K$ these tell us there are canonical isomorphisms
\begin{enumerate}
\item $\hat K_{k_1 k_2} \simeq \hat K_{k_1} \otimes \hat K_{k_2}$,\\
\item $\hat K_{\Ad_l (k)} \simeq \hat K_k$,\\
\item $\hat K_{k^{-1}} \simeq \hat K^*_{k}$,
\end{enumerate}
where $l \in L$ and $k \in K$. 

We will be concerned with bundles over $L^n \times K^m$ defined by (products and compositions of) pullbacks of the maps above. We will call such a bundle an \xm-bundle. More precisely, we make the following definition.

\begin{definition}\label{D:xm}
Let $f \colon L^n \times K^m \to K$ be a map given by composition of the following operations:
\begin{enumerate}
\item multiplication in $L$ and $K$;
\item inversion in $L$ and $K$;
\item inclusions $K \hookrightarrow L$ and $1 \hookrightarrow K$;
\item projections $L^n \times K^m \to L^p \times K^q$;
\item diagonals $L \times K \to L^p \times K^q$;
\item the adjoint action $\Ad\colon L \times K \to K$.
\end{enumerate}
We call $P \to L^n \times K^m$ an \emph{\xm-bundle} if $P \simeq f^{-1}(\hat K)$ for some $f$ as above. Additionally, we define an \emph{\xm-morphism} between \xm-bundles on the same base to be a map of bundles built from compositions and products of the three structural maps of the crossed module  (\ref{E:mult}) -- (\ref{E:inv}) above. An \xm-morphisms is clearly an isomorphism, since maps of principal bundles are so.

\end{definition}

We have the following result, which we will use repeatedly.

\begin{lemma}\label{L:uniqueness}
There exists at most one \xm-morphism between any two \xm-bundles.
\end{lemma}

\begin{proof}
To prove this we first make the following observation: Suppose $P$ is an \xm-bundle that is the pullback of a map $f \colon L^n \times K^m \to K$. We can factor $f$ through a product of $K$'s by leaving all the multiplication maps in $f$ until last; that is, we can write $f$ as a $k$-tuple $(f_0,\ldots, f_k) \colon L^n \times K^m \to K^k$, composed with the map $K^k \xrightarrow{m} K$ given by multiplication. The maps $f_1, \ldots, f_k$ do not contain among them any multiplication maps in $K$. Moreover, since $\Ad$ is a homomorphism we can further factorise the map $(f_0,\ldots, f_k)$ through $L^{\ell} \times K^{k}$ by leaving all the maps involving $\Ad$ until last; so $(f_0,\ldots, f_k)$ is given by a composition $L^n \times K^m \xrightarrow{g} L^{\ell} \times K^{k} \xrightarrow{a} K^k$, where we have denoted by $a$ the map involving all adjoints and by $g$ the map comprised of all other structure. As above, denote by $\hat K^{\hat\otimes k}$ the external tensor product of $\hat K$ with itself over $K^k$. Then the pullback diagram (\ref{E:Ad}) implies that the following diagram is a pullback
$$
\xymatrix{
L ^{\ell}\times \hat K^{\hat\otimes k} \ar[r] \ar[d]	& \hat K^{\hat\otimes k}\ar[d]\\
L^{\ell} \times K^{k} \ar^-{a}[r]		& K^k
}
$$
On the other hand, if we consider the map $L^\ell \times K^k \xrightarrow{\pr} K^k \xrightarrow{m} K$ given by projection and then multiplication, we see that the following diagram is also a pullback
$$
\xymatrix{
L ^{\ell}\times \hat K^{\hat\otimes k} \ar[r] \ar[d]	& \hat K^{\hat\otimes k} \ar[r] \ar[d]	& \hat K\ar[d]\\
L^{\ell} \times K^{k} \ar^-\pr[r]	& K^k \ar^m[r]	& K
}
$$
This gives the commutative diagram
$$
\xymatrix@C=7ex@R=3ex{
	&& \hat K^{\hat\otimes k} \ar[rr] \ar[ddd]|!{[ddl];[ddrr]}\hole	&& \hat K \ar[ddd]|!{[ddl];[ddr]}\hole&\\
P \ar[urr]	\ar[dr] \ar[ddd] &&&&&&\\
	& L^{\ell} \times \hat K^{\hat\otimes k} \ar[rr] \ar[ddd] \ar[uur]	&& \hat K^{\hat\otimes k} \ar[ddd] \ar[rr]&& \hat K \ar[ddd]\\
	&& K^{k} \ar^(0.45)m[rr]|!{[ur];[ddr]}\hole	&& K &\\
L^n \times K^m \ar^(0.4){(f_0, \ldots, f_k)}[urr]|!{[uur];[dr]}\hole	\ar_g[dr] &&&&&&\\
	& L^{\ell} \times K^{k} \ar_\pr[rr]  \ar_a[uur]	&& K^{k} \ar_m[rr]&& \hat K\\
}
$$

This means that an \xm-bundle $P$ given by a map $f$ is isomorphic (via a unique isomorphism) to a bundle pulled back by only the structure maps that do not include $\Ad$ (i.e.~along the bottom sequence of arrows in the diagram above). Now suppose that we have two \xm-bundles $P$ and $Q$, with an \xm-morphism between them. Then if we write them in the reduced form above (by which we mean they are pullbacks by maps not involving the adjoint action) there will be a unique \xm-morphism between them. Since there is a unique isomorphism from $P$ to its reduced form and from $Q$ to its reduced form, we have a unique \xm-morphism from $P$ to $Q$.
\end{proof}

What Lemma \ref{L:uniqueness} means is that if we have two \xm-bundles and we write out the fibres of each as 
$$
\hat K_{f_1(l_1, \ldots, l_n, k_1, \ldots, k_m)}   \otimes \hat K_{f_2(l_1, \ldots, l_n, k_1, \ldots, k_m)} \otimes \cdots \otimes \hat K_{f_k(l_1, \ldots, l_n, k_1, \ldots, k_m)},
$$
 then if the product of all the subscripts are equal once the maps involving $\Ad$ are removed, there is a unique \xm-morphism between them.

\subsection{Bundle 2-gerbes and trivialisations}

In \cite{Ste} the third author defines a notion related to the one developed in this paper; that of a \emph{bundle 2-gerbe}. The definition in \cite{Ste} is quite complicated, however we can make a useful simplification by employing the ideas from Section \ref{sec:ss&bg}, specifically Proposition \ref{prop:trivialisation descent}. This gives us the following equivalent definition of bundle 2-gerbe:

\begin{definition}\label{D:bundle 2-gerbe}
A \emph{bundle 2-gerbe} $(\cG, P)$ on $X$ consists of the following data:
\begin{enumerate}
	\item a surjective submersion $P \to X$;
	\item a bundle gerbe $\cG = (E, Y)$ over $P^{[2]}$;
	\item a trivialisation $M$ of $(\d(E), \mu^{-1}(Y)_2)$; and
	\item a section $a$ of $A_M \to P^{[4]}$ satisfying $\d (a) = 1$ as a section of $\d(A_M)$.
\end{enumerate}
\end{definition}
Here the bundle gerbe $(\d(E), \mu^{-1}(Y)_2)$ is the restriction of $\d(\cG)$ to the surjective submersion $\mu^{-1}(Y)_2 \to P^{[3]}$ as in Example \ref{ex:delta G for a bundle 2-gerbe}, and the line bundle $A_M \to P^{[4]}$ is the descent of $\d(M) \to \mu^{-1}(Y)_2$ as in Proposition \ref{prop:trivialisation descent}.

\begin{remark}
We leave it to the reader to show that this definition is equivalent to the one in \cite{Ste}. The main point is that the definition from \cite{Ste} involves a trivialisation $M$ of the bundle gerbe $(\d(E), \d(Y))$, and so $\d(M) \to \d^2(Y)$ does not descend as in our definition (since the bundle gerbe $\d(\d(E), \d(Y)) = (\d^2(E), \d^2(Y))$ is not strongly trivial). Therefore one needs to consider the difference of $\d(M)$ and the canonical trivialisation of $(\d^2(E), \d^2(Y))$. One then has a section of this and the appropriate conditions on this section.
\end{remark}

Specifically, we are interested in trivial bundle 2-gerbes. With the appropriate modifications the definition of a trivialisation of a bundle 2-gerbe  is

\begin{definition}
A bundle 2-gerbe $(\cG, P)$ over $X$ is \emph{trivial} if the following conditions are satisfied:
\begin{enumerate}
  \item there exists a bundle gerbe $\cH = (Q, Z)$ over $P$ and a stable isomorphism $L \colon  \cG \to \d (\cH)$, where $\d(\cH) = (\d(Q), \mu^{-1}(Z^2))$;
  \item there exists a section $\theta$ of the bundle $M\otimes \d (L)$ over $P^{[3]}$ satisfying $\d (\theta) = a$.
\end{enumerate}
The data $(\cH, L,\theta)$ will be called a \emph{trivialisation} of the bundle 2-gerbe.
\end{definition}

Note that if $\cG = (E, Y)$ then in (1) the bundle $L$ sits over the space  $Y \times_{P^{[2]}} \mu^{-1}(Z^2)$ and in (2) $M\otimes \d(L)$ sits over the space $\mu^{-1}(Y)_2 \times_{P^{[3]}} \mu^{-1}(Z^3)$ but descends to $P^{[3]}$. The equation $\d (\theta) = a$ makes sense because $\d(\theta)$ is a section of $\d(M\otimes \d(L)) = \d(M)$, viewed as a bundle over $P^{[4]}$, which is the bundle $A_M$.

\subsection{The Chern--Simons bundle 2-gerbe}\label{SS:CS}

The example of a bundle 2-gerbe in which we are interested is the \emph{Chern--Simons bundle 2-gerbe} \cite{CarJohMur} associated to a principal $G$-bundle $P \to X$. This is defined by taking the simplicial manifold $P^{[\bullet+1]}$ and using the isomorphism $P^{[2]} = P \times G$ to pull back the basic gerbe on $G$ to $P^{[2]}$. The model of the basic bundle gerbe that we use here is different to that in Section \ref{S:basic}; it is the lifting bundle gerbe for the path fibration of the group $G$. The lifting bundle gerbe was introduced in \cite{Mur} and the example of the path fibration of a compact Lie group $G$ was studied in detail in \cite{MurSte03}. It is given by taking the surjective submersion $PG \to G$, which is a principal $\Omega G$-bundle, and identifying $PG^{[2]}$ with $PG \times \Omega G$. We then pull back the central extension $\widehat{\Omega G}$ by the projection $PG \times \Omega G \to \Omega G$.

Next we give the data of the Chern--Simons bundle 2-gerbe in detail. According to the description above, it is the pullback of the basic gerbe on $G$. Thus we have the following depiction.
$$
\xymatrix@R=1.5ex@C=2ex{
  & & PG \times \widehat{\Omega G}\ar[dd] &\\
  P \times PG \times \widehat{\Omega G}\ar[dd]  & & &\\
  & & PG \times \Omega G\ar@<0.75ex>[]!DR;[dr]\ar@<-0.75ex>[]!DR;[dr]    &\\
  P \times PG \times \Omega G\ar@<0.75ex>[]!DR;[dr]!UL\ar@<-0.75ex>[]!DR;[dr]!UL  & & &PG\ar[dd]\\
  & P \times PG\ar[dd]  & &\\
  & & &G  \\
  & P \times G\ar[urr]\ar@<0.75ex>[]!DR;[dr]\ar@<-0.75ex>[]!DR;[dr] & &\\
  & & P \ar[dd] &\\
  &&&\\
  & & X&  
}
$$

We shall denote the bundle gerbe over $P\times G$ by $\cG$.

In order to describe the rest of the data defining the Chern--Simons bundle 2-gerbe as per Definition \ref{D:bundle 2-gerbe}, we need to consider the semi-simplicial surjective submersion $\mu^{-1}(P \times PG)_\bullet \to P^{[\bullet+1]}$ and the bundle gerbe 
$$
\d(\cG) = (\d(P \times PG \times \widehat{\Omega G}), \mu^{-1}(P \times PG)_2).
$$
Calculation shows that the low dimensional spaces in the semi-simplicial surjective submersion $\mu^{-1}(P \times PG)_\bullet \to P^{[\bullet + 1]}$ are
$$
 \xymatrix{
P\times PG^3\times {\Omega G}^3 \ar[d] \ar@<-1.2ex>[r]\ar@<-0.4ex>[r]\ar@<0.4ex>[r]\ar@<1.2ex>[r] &    P\times PG^2\times {\Omega G} \ar[d] \ar@<-1ex>[r]\ar[r]\ar@<1ex>[r] & P\times PG \ar[d] \ar@<0.6ex>[r] \ar@<-0.6ex>[r] & P \ar@{=}[d]\\
P \times G^3 \ar@<-1.2ex>[r]\ar@<-0.4ex>[r]\ar@<0.4ex>[r]\ar@<1.2ex>[r]  &    P \times G^2 \ar@<-1ex>[r]\ar[r]\ar@<1ex>[r] & P\times G \ar@<0.6ex>[r] \ar@<-0.6ex>[r] & P
  }
 $$
The maps $d_i \colon \mu^{-1}(P \times PG)_2 = P\times PG^2\times {\Omega G} \to P \times PG$ are given by 
\begin{equation}\label{E:d012}
\begin{aligned}
	d_0(p, \gamma_1, \gamma_2, \omega) 	&= (p \gamma_1(1), \gamma_2), \\
	d_1(p, \gamma_1, \gamma_2, \omega) 	&= (p, \gamma_1\gamma_2\omega),\\
	d_2(p, \gamma_1, \gamma_2, \omega) 	&= (p, \gamma_1),
\end{aligned}
\end{equation}
and the maps $d_i \colon \mu^{-1}(P \times PG)_3 = P\times PG^3\times {\Omega G}^3 \to P\times PG^2\times {\Omega G}$ are given by
\begin{equation}\label{E:d0123}
\begin{aligned}
	d_0(p, \gamma_1, \gamma_2, \gamma_3, \omega_1, \omega_2, \omega_3) 	&= (p \gamma_1(1), \gamma_2, \gamma_3 , \omega_3), \\
	d_1(p, \gamma_1, \gamma_2, \gamma_3,  \omega_1, \omega_2, \omega_3) 	&= (p, \gamma_1\gamma_2\omega_1, \gamma_3, \Ad_{\gamma_3^{-1}}(\omega_1^{-1}) \omega_2),\\
	d_2(p, \gamma_1, \gamma_2, \gamma_3,  \omega_1, \omega_2, \omega_3) 	&= (p, \gamma_1, \gamma_2\gamma_3\omega_3, \omega_3^{-1}\omega_2),\\
	d_3(p, \gamma_1, \gamma_2, \gamma_3, \omega_1, \omega_2, \omega_3) 	&= (p, \gamma_1, \gamma_2, \omega_1).
\end{aligned}
\end{equation}
Note that 
$$
\mu^{-1}(P \times PG)_2^{[2]} = (P \times PG^2 \times \Omega G) \times_{P \times G^2} (P \times PG^2 \times \Omega G) = P \times PG^2 \times \Omega G^4
$$
via the projections $(p, \gamma_1 ,\gamma_2, \omega_0, \omega_1, \omega_2, \omega_3) \mapsto (p, \gamma_1, \gamma_2, \omega_0)$ and $(p, \gamma_1 \omega_1, \gamma_2 \omega_2, \omega_3)$. The maps $\mu^{-1}(P \times PG)_2^{[2]} = P \times PG^2 \times \Omega G^4 \longrightthreearrow P \times PG \times \Omega G = (P \times PG)^{[2]}$ are given by
\begin{align*}
	d_0(p, \gamma_1, \gamma_2, \omega_0, \omega_1, \omega_2, \omega_3) 	&= (p \gamma_1(1), \gamma_2, \omega_2), \\
	d_1(p, \gamma_1, \gamma_2, \omega_0, \omega_1, \omega_2, \omega_3) 	&= (p, \gamma_1\gamma_2\omega_0, \omega_0^{-1} \Ad_{\gamma_2^{-1}} (\omega_1)\omega_2 \omega_3),\\
	d_2(p, \gamma_1, \gamma_2, \omega_0, \omega_1, \omega_2, \omega_3) 	&= (p, \gamma_1, \omega_1).
\end{align*}
Therefore, the bundle gerbe $\d(\cG)$ is given by the line bundle $E \to P \times PG^2 \times \Omega G^4$ whose fibre at $(p, \gamma_1, \gamma_2, \omega_0, \omega_1, \omega_2, \omega_3)$ is
\begin{align*}
E_{(p, \gamma_1, \gamma_2, \omega_0, \omega_1, \omega_2, \omega_3)}
	&\simeq\widehat{\Omega G}_{\omega_2} \otimes \widehat{\Omega G}^*_{\omega_0^{-1} \Ad_{\gamma_2^{-1}} (\omega_1)\omega_2 \omega_3} \otimes \widehat{\Omega G}_{\omega_1}\\
	&\simeq \widehat{\Omega G}_{\omega_2} \otimes \widehat{\Omega G}^*_{\omega_0^{-1}}\otimes \widehat{\Omega G}^*_{ \Ad_{\gamma_2^{-1}} (\omega_1)}\otimes \widehat{\Omega G}^*_{\omega_2}\otimes \widehat{\Omega G}^*_{ \omega_3} \otimes \widehat{\Omega G}_{\omega_1}\\
	&\simeq \widehat{\Omega G}_{\omega_2} \otimes \widehat{\Omega G}_{\omega_0}\otimes \widehat{\Omega G}_{ \omega_1^{-1}}\otimes \widehat{\Omega G}_{\omega_2^{-1}}\otimes \widehat{\Omega G}_{ \omega_3^{-1}} \otimes \widehat{\Omega G}_{\omega_1}\\
	&\simeq \widehat{\Omega G}_{\omega_0}\otimes\widehat{\Omega G}_{ \omega_3^{-1}}.
\end{align*}
The isomorphisms above are the unique ones guaranteed by Lemma \ref{L:uniqueness}.
This calculation shows that the bundle gerbe $\d(\cG)$ is trivial with trivialisation $M = P \times PG^2 \times \widehat{\Omega G}^* \to P \times PG^2 \times \Omega G$, which is the data of Definition \ref{D:bundle 2-gerbe} (3).

For \ref{D:bundle 2-gerbe} (4) we note that the bundle $\d(M) \to P \times PG^3 \times \Omega G^3$ has fibre at the point $(p, \gamma_1, \gamma_2, \gamma_3, \omega_1, \omega_2, \omega_3)$ given by 
$$
\widehat{\Omega G}_{\omega_3}^*\otimes\widehat{\Omega G}_{ \Ad_{\gamma_3^{-1}}(\omega_1^{-1}) \omega_2}\otimes \widehat{\Omega G}_{\omega_3^{-1}\omega_2}^*\otimes\widehat{\Omega G}_{ \omega_1},
$$
and so we have the (unique) sequence of isomorphisms
\begin{align*}
\d(M)_{(p, \gamma_1, \gamma_2, \gamma_3, \omega_1, \omega_2, \omega_3)}
	&\simeq\widehat{\Omega G}_{\omega_3}^*\otimes\widehat{\Omega G}_{ \Ad_{\gamma_3^{-1}}(\omega_1^{-1}) \omega_2}\otimes \widehat{\Omega G}_{\omega_3^{-1}\omega_2}^*\otimes\widehat{\Omega G}_{ \omega_1}\\
	&\simeq \widehat{\Omega G}_{\omega_3}^*\otimes\widehat{\Omega G}_{ \Ad_{\gamma_3^{-1}}(\omega_1^{-1})} \otimes \widehat{\Omega G}_{ \omega_2}\otimes \widehat{\Omega G}_{\omega_3^{-1}}^* \otimes \widehat{\Omega G}_{\omega_2}^*\otimes\widehat{\Omega G}_{ \omega_1}\\
	&\simeq\widehat{\Omega G}_{\omega_3}^*\otimes\widehat{\Omega G}_{ \omega_1^{-1}} \otimes \widehat{\Omega G}_{ \omega_2}\otimes \widehat{\Omega G}_{\omega_3^{-1}}^* \otimes \widehat{\Omega G}_{\omega_2}^*\otimes\widehat{\Omega G}_{ \omega_1}\\
	&\simeq\widehat{\Omega G}_{\omega_3}^*\otimes\widehat{\Omega G}_{ \omega_1}^* \otimes \widehat{\Omega G}_{ \omega_2}\otimes \widehat{\Omega G}_{\omega_3} \otimes \widehat{\Omega G}_{\omega_2}^*\otimes\widehat{\Omega G}_{ \omega_1}\\
	&\simeq U(1).
\end{align*}
Therefore we have a trivialisation of $\d(M)$ and hence $A_M$, by Lemma \ref{L:uniqueness}. We define the section $a$ from Definition \ref{D:bundle 2-gerbe} (4) to be this trivialisation.

The situation is summarised in the following diagram
$$
\xymatrix@C=0.3ex{
	&	&	&	&\d\left(P \times PG \times \widehat{\Omega G}\right) \ar[d]	&	&P \times PG \times \widehat{\Omega G}	\ar[d]	\\
A_M \ar@/_5ex/[ddr]	&\d(M)\ar[d]	&	&M \ar[d]	&P \times PG^2 \times \Omega G^4 \ar@<0.75ex>[]!DL;[dl] \ar@<-0.75ex>[]!DL;[dl] \ar@<-1ex>[rr]\ar[rr]\ar@<1ex>[rr] 	&	&P \times PG \times \Omega G	 \ar@<1.5ex>[]!D;[dl]+UR \ar[]!D;[dl]+UR	\\
	&P \times PG^3 \times \Omega G^3 \ar@<-1.2ex>[rr]\ar@<-0.4ex>[rr]\ar@<0.4ex>[rr]\ar@<1.2ex>[rr] \ar[d]	&	&P \times PG^2 \times \Omega G \ar@<-1ex>[rr]\ar[rr]\ar@<1ex>[rr] \ar[d]	&	&P \times PG \ar[d]	&		\\
	&P \times G^3 \ar@<-1.2ex>[rr]\ar@<-0.4ex>[rr]\ar@<0.4ex>[rr]\ar@<1.2ex>[rr]	&	&P \times G^2 \ar@<-1ex>[rr]\ar[rr]\ar@<1ex>[rr]	&	&P \times G \ar@<0.6ex>[r] \ar@<-0.6ex>[r] 	&P\ar[d]	\\
	&&&&&&X
}
$$

\subsection{The simplicial extension of a string structure}\label{SS:simplicial extension}
 
Suppose now that we have a trivialisation $\cH$ of the Chern--Simons bundle 2-gerbe associated to the $G$-bundle $P \to X$. We will build a simplicial extension of $\cH$ over the simplicial manifold $E\cK(P)_\bullet$ associated to the action of the crossed module $\cK = (\widehat{\Omega G} \to PG)$ from Example \ref{ex:string} on $P$. We will now describe this simplicial manifold in more detail.

Given a crossed module $\hat K \xrightarrow{t} L$ of Lie groups acting on a manifold $P$, one can form an action \emph{2-groupoid}. This has as objects the manifold $P$, as 1-arrows the product $P \times L$, and as 2-arrows the product $P \times L\times \hat K$:
\[\xymatrix@C=6ex{
  P \times L \times \hat K \ar@<0.6ex>[r]^-{\pr_{12}} \ar@<-0.6ex>[r]_-{1_P\times f} & P \times L  \ar@<0.6ex>[r]^-{\pr_{1}} \ar@<-0.6ex>[r]_-{ \text{act}} & P 
}\]
where the action $f$ of $\hat K$ on $L$ is via the map $t$ and the action of $L$ on $P$ is part of the definition of the action of $\cK$ on $P$.
The precise description of the structural maps (i.e.\ sources, targets and compositions) of this 2-groupoid we shall leave to the reader as an instructive exercise, since we are more interested in the nerve of this 2-groupoid (as defined by Street \cite{Street} and Duskin \cite{Duskin}), which we shall describe explicitly in low dimensions. This nerve is what we have called $E\cK(P)_\bullet$.

The intuitive picture that the reader should keep in mind is that of the nerve of the action 1-groupoid, but instead of commuting triangles making up the dimension 2 faces of simplices, one should fill it with an element of the group $\hat K$. 
A 2-simplex is thus a triangle commuting up to a 2-arrow; a 3-simplex is a tetrahedron with faces labelled as such as commuting in the 2-dimensional sense. 
Table \ref{table:nerve_of_action_2-groupoid} specifies the face maps that we shall need in the course of this section.

\begin{table}
  \centering
  \[\xymatrix@C=4ex{
  E\cK(P)_\bullet =  \cdots  \ P \times L^3 \times \hat K^3  \ar@<-1.2ex>[r]\ar@<-0.4ex>[r]\ar@<0.4ex>[r]\ar@<1.2ex>[r]  &P\times L^2 \times \hat K  \ar@<-1ex>[r]\ar[r]\ar@<1ex>[r]  &P \times L \ar@<0.6ex>[r] \ar@<-0.6ex>[r] &P
    }
  \]
  \begin{tabular}{l|c}
    \hline\\
    $\xymatrix@C=3ex{E\cK(P)_1 \ar@<0.6ex>[r] \ar@<-0.6ex>[r] &E\cK(P)_0}$   & $\begin{aligned} d_0(p,l) &= pl \\ d_1(p,l) &= p \end{aligned}$ \\
              \\
    \hline\\
    $\xymatrix@C=3ex{E\cK(P)_2 \ar@<-1ex>[r]\ar[r]\ar@<1ex>[r] &E\cK(P)_1}$  & 	$\begin{aligned} d_0(p,l_1,l_2,k) &= (pl_1,l_2)\\
						d_1(p,l_1,l_2,k) &= (p,l_1l_2t(k))\\
						d_2(p,l_1,l_2,k) &= (p,l_1) \end{aligned}$\\ 
                \\
    \hline\\
    $\xymatrix@C=3ex{E\cK(P)_3 \ar@<-1.2ex>[r]\ar@<-0.4ex>[r]\ar@<0.4ex>[r]\ar@<1.2ex>[r]  &E\cK(P)_2}$
		& $\begin{aligned} d_0(p,l_1,l_2,l_3,k_1,k_2,k_3) &= (p l_1, l_2, l_3, k_3)\\
		d_1(p,l_1,l_2,l_3,k_1,k_2,k_3) &= (p, l_1 l_2 t(k_1), l_3, k_2)\\[-5pt] 
		d_2(p,l_1,l_2,l_3,k_1,k_2,k_3) &= (p, l_1, l_2 l_3 t(k_3), k_3^{-1} (k_1^{l_3^{-1}}) k_2)\\
		d_3(p,l_1,l_2,l_3,k_1,k_2,k_3) &= (p, l_1, l_2, k_1) \end{aligned}$\\
                \\
    \hline
  \end{tabular}\\[10pt]
  \caption{The nerve of the action 2-groupoid in low dimensions} 
  \label{table:nerve_of_action_2-groupoid}
\end{table}

The crossed module we are interested in is $\widehat{\Omega G} \to PG$, from Example \ref{ex:string}, which gives rise to the String group of $G$. As described earlier, it acts naturally on the $G$-bundle $P$ via the map to $G$.
For the simplicial manifold $E\cK(P)_\bullet$ arising from this action there is a simplicial map $e\colon E\cK(P)_\bullet \to P^{[\bullet+1]}$, given by evaluating all paths at their endpoints, and forgetting factors of $\widehat{\Omega G}$. 
In low degrees this is
\[
  \xymatrix{
\cdots \ar@<-1.4ex>[r] \ar@<-0.7ex>[r]\ar[r]\ar@<0.7ex>[r]\ar@<1.4ex>[r]& P\times PG^3\times \widehat{\Omega G}^3 \ar[d]_{e_3 = \id \times \ev_1^3} \ar@<-1.2ex>[r]\ar@<-0.4ex>[r]\ar@<0.4ex>[r]\ar@<1.2ex>[r] &    P\times PG^2\times \widehat{\Omega G} \ar[d]_{e_2 = \id \times \ev_1^2} \ar@<-1ex>[r]\ar[r]\ar@<1ex>[r] & P\times PG \ar@<0.6ex>[r] \ar@<-0.6ex>[r] \ar[d]_{e_1 = \id \times \ev_1} & P \ar@{=}_{e_0 = \id}[d] \\
\cdots \ar@<-1.4ex>[r]\ar@<-0.7ex>[r]\ar[r]\ar@<0.7ex>[r]\ar@<1.4ex>[r]& P \times G^3 \ar@<-1.2ex>[r]\ar@<-0.4ex>[r]\ar@<0.4ex>[r]\ar@<1.2ex>[r]  &    P \times G^2 \ar@<-1ex>[r]\ar[r]\ar@<1ex>[r] & P\times G \ar@<0.6ex>[r] \ar@<-0.6ex>[r] & P
  }
\]

We shall denote the operation $\d$ for the simplicial manifolds $E\cK(P)_\bullet$ and $P^{[\bullet+1]}$ by $\delta_{E\cK}$ and $\delta_P$, respectively.  Thus the definition of the Chern--Simons bundle 2-gerbe and a trivialisation of it involve $\delta_P$ everywhere, while a simplicial extension of $\cH$ over $E\cK(P)_\bullet$ will involve $\d_{E\cK}$. We also remind the reader that if $\cG = (E, Y)$ over $P^{[2]}$ and $\cH = (Q, Z)$ over $P$ then by $\d^k(\cG)$ we mean $(\d^k(E), \mu^{-1}(Y)_{k})$ and by $\d^k(\cH)$ we mean $(\d^k(Q), \mu^{-1}(Z^{k+1}))$.

Given a trivialisation $(\cH,L,\theta)$ of the Chern--Simons bundle 2-gerbe, we will construct a simplicial extension of $\mathcal{H}$ over $E\cK(P)_\bullet$ by pulling back the data of the bundle 2-gerbe $(\cG, P )$ along $e$.

We construct the trivialisation $T$ of $ \d_{E\cK}(\cH) $ as follows. Notice first, since $e$ is a simplicial map, we have that $e^{-1} (\d_P (\cH))$ is canonically isomorphic to $\d_{E\cK}( \cH)$. We therefore have the stable isomorphism $e_1^{-1}(L) \colon e_1^{-1} (\cG) \to e_1^{-1}(\d_P(\cH)) = \d_{E\cK}(\cH)$. To construct $T$ we combine this with a trivialisation of $e_1^{-1}(\cG)$ using the following lemma.

\begin{lemma}\label{L:object space}
Let $(Q,Y)$ be a bundle gerbe over a manifold $X$ with surjective submersion $\pi \colon Y \to X$. Then $\pi^{-1}(Q, Y)$ (as a bundle gerbe over $Y$) has a canonical trivialisation given by $\tau = Q$.
\end{lemma}

\begin{proof}
First notice that the statement makes sense because $\pi^{-1}(Q, Y)$ has as surjective submersion the pullback $\pi^{-1}(Y) = Y^{[2]}$, and $\tau = Q$ is a line bundle over $Y^{[2]}$. The fibre product $Y^{[2]} \times_Y Y^{[2]}$ is given by $Y^{[3]}$ and the face maps are projection onto the first and second, and first and third factors, respectively. Thus we have $\pi^{-1}(Q)_{(y_1, y_2, y_3)} = Q_{(y_2, y_3)} = Q_{(y_2, y_1)} Q_{(y_1, y_3)} = Q_{(y_1, y_2)}^*  Q_{(y_1, y_3)}$, and so $\pi^{-1}(Q)$ is trivialised by $\tau = Q$.
\end{proof}

Now, $e_1 \colon P\times PG \to P \times G$ is the surjective submersion for the bundle gerbe $\cG$, so Lemma \ref{L:object space} gives us the trivialisation $\tau$ of $e_1^{-1}(\cG)$. Thus we have the trivialisation $T = \tau \otimes e_1^{-1}(L)$ of $\d_{E\cK}(\cH)$.

Next we need a section $s$ of $A_T = \d_{E\cK} (\tau \otimes e_1^{-1}(L) )$ over $E\cK(P)_2 = P \times PG^2 \times \widehat{\Omega G}$. Since $\cH$ is a trivialisation of the Chern--Simons bundle 2-gerbe we have a section $\theta$ of $M \otimes \d_P (L)$ over $P \times G^2$. We claim that $e_2^{-1}(M \otimes \d_P(L)) $ is canonically isomorphic to $\d_{E\cK}(\tau \otimes e_1^{-1}(L))$ and therefore we can define $s = e_2^{-1}(\theta)$. We have 
$$
e_2^{-1}(M \otimes \d_P(L))  = e_2^{-1}(M) \otimes e_2^{-1}( \d_P(L))  = e_2^{-1}(M) \otimes \d_{E\cK}(e_1^{-1}(L) ).
$$
So we need to show that $e_2^{-1}(M)$ is canonically isomorphic to $\d_{E\cK}(\tau)$. Notice that $P \times PG^2 \times \widehat{\Omega G}$ is the total space of the dual of $M$. We have the following result.

\begin{lemma}\label{L:total space triv}
Let $(Q,Y)$ be a bundle gerbe on a manifold $X$, with surjective submersion $\pi \colon Y \to X$. Let $\hat{\pi} \colon R\to Y$ be a trivialisation of $(Q,Y)$, and $R^*$ its dual. Denote by $p$ the composite map $\pi \circ \hat{\pi} \colon R^* \to Y \to X$. Recall that, by Lemma \ref{L:object space}, the bundle gerbe $\pi^{-1}(Q, Y)$ has a trivialisation $\tau$. Then the two trivialisations $p^{-1}(R)$ and $\hat{\pi}^{-1}(\tau^*)$ of $p^{-1}(Q, Y)$ are canonically isomorphic.
\end{lemma}

\begin{proof}
Recall from the proof of Lemma \ref{L:object space} that the bundle gerbe $\pi^{-1}(Q, Y)$ over $Y$ has as surjective submersion $Y^{[2]} \to Y$. It has two trivialisations; $\tau$ and $\pi^{-1}(R)$. We have $\tau \oslash \pi^{-1}(R) = R^*$ since, for $y \in Y$ and any $x$ in the same fibre, $(\tau\otimes \pi^{-1}(R)^*)_{(y, x)} = Q_{(y, x)} R_x = R_y^*$, using the fact that $R$ is a trivialisation. We have $\hat{\pi}^{-1}(\tau) \oslash p^{-1}(R) = \hat{\pi}^{-1}(\tau\oslash \pi^{-1}(R)) = \hat{\pi}^{-1}(R^*)$, which is canonically trivial. Therefore $\hat{\pi}^{-1}(\tau)$ is canonically isomorphic to $p^{-1}(R)$.
\end{proof}

To apply Lemma \ref{L:total space triv} to the bundle gerbe $\d_P(\cG)$ with its trivialisation $M$ notice that $e_2$ factors as $P \times PG^2 \times \widehat{\Omega G} \xrightarrow{\hat{\pi}} P \times PG^2 \times {\Omega G} = \mu^{-1}(P \times PG)_2 \xrightarrow{\pi} P \times G^2$ such that the following diagram commutes
\[
  \xymatrix{
P\times PG^2\times \widehat{\Omega G} \ar[d]_{\hat \pi} \ar@<-1ex>[r]\ar[r]\ar@<1ex>[r] \ar@/_5ex/!DL;[dd]_{e_2}  & P\times PG \ar@<0.6ex>[r] \ar@<-0.6ex>[r] \ar@{=}[d] & P \ar@{=}[d] \\
P\times PG^2\times {\Omega G} \ar[d]_{\pi} \ar@<-1ex>[r]\ar[r]\ar@<1ex>[r] &  P\times PG \ar@<0.6ex>[r] \ar@<-0.6ex>[r] \ar[d]_{e_1} & P \ar@{=}[d] \\
P \times G^2 \ar@<-1ex>[r]\ar[r]\ar@<1ex>[r] & P\times G \ar@<0.6ex>[r] \ar@<-0.6ex>[r] & P
  }
\]
and so the trivialisation $e_2^{-1}(M)$ is canonically isomorphic to $\hat \pi^{-1}(\tau_{\d(\cG)})$, where $\tau_{\d(\cG)}$ is the canonical trivialisation of $\pi^{-1}(\d_P(\cG))$ given by Lemma \ref{L:object space}. But $\tau_{\d(\cG)}$ is isomorphic to $\d_P(\tau)$, where $\tau$ is the canonical trivialisation of $e_1^{-1}(\cG)$. Hence the pullback $\hat{\pi}^{-1} (\tau_{\d(\cG)}) = \hat{\pi}^{-1} (\d_P(\tau))$ is isomorphic to $\d_{E\cK}$ applied to the dual of the canonical trivialisation of $e_1^{-1}(\cG)$, which is precisely $\d_{E\cK}(\tau)$. This allows us to define the section $s$ as the pullback of $\theta$ by $e_2$.

It only remains to show that $\d_{E\cK} (s) = 1$ as a section of $\d_{E\cK}(A_T)$, which is the descent of the bundle $\d_{E\cK}^2(\tau \otimes e_1^{-1}(L))$ to $E\cK(P)_ 3= P \times PG^3 \times \widehat{\Omega G}^3$.
Notice that since $(\cH, L, \theta)$ is a trivialisation of the Chern--Simons bundle 2-gerbe we have $\d_P(\theta) = a$, where $a$ is the section of $A_M$ from Section \ref{SS:CS}. Further, since $s = e_2^{-1}(\theta)$, we have $\d_{E\cK} (s) = \d_{E\cK} (e_2^{-1}(\theta)) = e_3^{-1}(\d_P(\theta)) = e_3^{-1}(a)$. 
Therefore, we need only show that $e_3^{-1}(a)$ is the canonical trivialisation of $\d_{E\cK}^2(\tau \otimes e_1^{-1}(L))$ under the isomorphism induced by Lemma \ref{L:total space triv}. 
In fact, since $a$ is a trivialisation of $A_M = \d_P(M) = \d_P(M) \otimes \d_P^2(L)$, the section $\d_P(\theta)$ induces the canonical trivialisation of $\d_P^2(L)$ and so it suffices to check that $e_3^{-1}(a)$ induces the canonical trivialisation of $\d_{E\cK}^2(\tau)$. 
Notice however, that the canonical trivialisation of $\d_{E\cK}^2(\tau )$ involves pairing up factors of $\widehat{\Omega G}$ and $\widehat{\Omega G}^*$, whereas $a$ is an \xm-morphism. To see that these are the same consider a crossed module $\hat K \xrightarrow{t} L$. We can factorize the constant map $K \to K;\, x \mapsto 1$ as 
$$
K \xrightarrow{\Delta} K \times K \xrightarrow{1 \times i} K \times K \xrightarrow{m} K,
$$
where $\Delta$ is the diagonal map, and $i$ and $m$ are inversion and multiplication in $K$, respectively. Then $(m \circ (1 \times i) \circ \Delta)^{-1}(K)$ is canonically trivial. However, we also have that $(m \circ (1 \times i) \circ \Delta)^{-1}(K)$ is isomorphic to $\hat K \otimes \hat K^*$, which is canonically trivial. We have the following trivial result

\begin{lemma}\label{L:canonical section}
The two trivialisations of $\hat K \otimes \hat K^*$ given above are equal.
\end{lemma}

The point is that both the canonical trivialisation of $\d_{E\cK}^2(\tau )$ (by Lemma \ref{L:canonical section}) and the trivialisation given by $e_3^{-1}(a)$ are \xm-morphisms as in Definition \ref{D:xm}, and therefore are equal by Lemma \ref{L:uniqueness}. So Lemma \ref{L:canonical section} tells us that the section $e_3^{-1}(a)$ (and hence $\d_{E\cK}(s)$) agrees with the canonical section of $\d_{E\cK}^2(\tau \otimes e_1^{-1}(L))$ (and hence $\d_{E\cK}(A_T)$). Therefore we have our main result 

\begin{theorem}\label{th:string}
Let $P \to X$ be a principal $G$-bundle and let $(\cH, L, \theta)$ be a trivialisation of the Chern--Simons bundle 2-gerbe of $P$. Then $\cH$ has a simplicial extension over the nerve of the action 2-groupoid of the induced String group action on $P$, given by $(T,s)$ constructed above.
\end{theorem}


\appendix

\section{Descent for trivialisations}
\begin{proposition}
\label{prop:descent-trivialisations}
Assume that $(P, Y)$ is a bundle gerbe over $M$ and that $\phi \colon X \to Y$ is morphism 
of surjective submersions over $M$.  Then if $T \to Y$ is a trivialisation of $(\phi^{-1}(P), X)$ there
is a trivialisation $\phi(T) \to Y$ with the property that $\phi^{-1}(\phi(T)) \to X$ is isomorphic to $T$ as a trivialisation of $(\phi^{-1}(P), Y)$.
\end{proposition}
\begin{proof}
Recall that a trivialisation $R \to Y$ of $(P,Y)$ is an isomorphism $P \to  \d_Y(R)$ which commutes with the bundle gerbe product on $P$ and the trivial bundle gerbe product on $ \d_Y(R)$. 
It is convenient to formulate this in the following way.  For $(y_1, y_2) \in Y^{[2]}$
we have an isomorphism 
\begin{align*}
R_{y_1} &\otimes P_{(y_1, y_2)}  \to R_{y_2}\\
r_1 &\otimes p_{12}  \mapsto r_1 p_{12} 
\end{align*}
and we require that for any $y_1, y_2, y_3 $ we have $(r_1 p_{12}) p_{23} = 
r_1 (p_{12} p_{23})$ where $p_{ij} \in P_{(y_i ,y_j)}$ and $(p_{12} p_{23})$ denotes the
bundle gerbe product. 

So if $T \to X$ is a trivialisation of $\phi^{-1}(P) \to X^{[2]}$ then 
we have 
$$
T_{x_1} P_{(\phi(x_1) ,\phi(x_2))} \to T_{x_2} ,
$$ 
with the corresponding condition on compatibility with the bundle gerbe product. 
We define $S \to X \times_M Y$ by $S_{(x, y)} = T_x \otimes P_{(\phi(x), y) }$. 
We want to show that $S$ descends to a bundle $\phi(T) \to Y$ and to this end we define  
$$
\phi_{x_2 x_1} \colon S_{x_1, y} \to S_{x_2 ,y}
$$ 
by $\phi_{x_2 x_1}( t_1 \otimes q_1 ) = (t_1 p_{12}) \otimes (p_{12}^* q_1)$
where $t_1 \in T_{x_1}$, $q_1 \in P_{(\phi(x) ,y) }$ and the definition involves
the choice of $p_{12} \in P_{(\phi(x_1), \phi(x_2))}$.  It is clearly independent of this choice and the choices representing the element in $S_{(x_1 ,y)}$.   We need 
to check that $\phi_{x_3 x_2} \phi_{x_2 x_1} = \phi_{x_3 x_1}$ and making 
appropriate choices of elements in the various spaces we have
\begin{align*}
\phi_{x_3 x_2} \phi_{x_2 x_1} ( t_1 \otimes q_1) & = \phi_{x_3 x_2}(t_1 p_{12}) \otimes (p_{12}^* q_1) \\
&= (t_1 p_{12} ) p_{23} \otimes p_{23}^* ( p_{12}^* q_1 ) \\
& = \phi_{x_3 x_1} ( t_1 \otimes q_1 )
\end{align*}
as required. 

Now define 
$$
\rho_{x  y_1 y_2} \colon S_{(x ,y_1)} \otimes P_{(y_1, y_2)} \to S_{(x, y_2)}
$$
by $\rho_{x y_1 y_2}( s_1\otimes q_{12}) = t_1 \otimes q_1 q_{12}$ where $s_1 = t_1 \otimes q_1$.
We want to show that 
$$ 
\xymatrixcolsep{5pc}\xymatrix{ 
S_{(x_1 ,y_1)} \otimes P_{(y_1, y_2)} \ar[d]_{\phi_{x_2 x_1} \otimes 1}  \ar[r]^-{\rho_{x_1 y_1 y_2}} & S_{(x_1 ,y_2)} \ar[d]^{\phi_{x_2 x_1}} \\ 
S_{(x_2 ,y_1)} \otimes P_{(y_1 ,y_2)} \ar[r]_-{\rho_{x_2 y_1 y_2}} & S_{(x_2 ,y_2)} } 
$$ 
commutes. To see this we note that if $s_1 = t_1 \otimes q_1$ then 
\begin{align*}
\rho_{x_2 y_1 y_2} \circ (\phi_{x_2x_1} \otimes 1)  (s_1 \otimes q_{12})& = \rho_{x_2 y_1 y_2} ( (t_1 p_{12} \otimes p_{12}^* q_1) \otimes q_{12}) \\
& = t_1 p_{12} \otimes ((p_{12}^* q_1) q_{12})\\
& = t_1 p_{12} \otimes p_{12}^* (q_1 q_{12})\\
& = \phi_{(x_2x_1)}(t_1 \otimes (q_1 q_{12})) \\
& = \phi_{(x_2x_1)} \circ \rho_{x_1 y_1 y_2}(s_1 \otimes q_{12}).
\end{align*}
Hence this map descends to give an isomorphism
$$
\phi(T)_{y_1} \otimes P_{(y_1 ,y_2)} \to \phi(T)_{y_2}
$$
which we write as $s_1 \otimes q_{12} \mapsto s_1 q_{12} $ and we have to 
check that $(s_1 q_{12}) q_{23} =  s_1 (q_{12} q_{23})$. We have
$(s_1 q_{12}) q_{23} = (t_1 \otimes q_1 q_{12} ) q_{23} = 
t_1 \otimes (q_1 q_{12} ) q_{23} = t_1 \otimes q_1 (q_{12} q_{23} )
= s_1 (q_{12} q_{23})$,
as required.

Finally notice that the pullback of $\phi(T)$ is 
$$
\phi^{-1}(\phi(T))_x = \phi(T)_{\phi(x)} = S_{(x, \phi(x))} = 
T_x P_{(\phi(x), \phi(x))} = T_x,
$$
as required. 
\end{proof}

\section{Calculations supporting the proof of Proposition \ref{thm:equivariant gerbe on U(n)}} \label{app:long calc for theorem 5.2}

We prove the equation $\delta(f) - d\beta = \pi^*(\omega)$.  Our strategy, as in 
\cite{MurSte}, is to transfer the problem to the more convenient space $G/T\times Y_T$, where 
$T$ is the subgroup of diagonal matrices in $G = U(n)$ and $Y_T = (T\times Z)\cap Y$.  Recall the 
canonical map $p_Y\colon G/T\times Y_T\to Y$ defined by $(gT,(t,z)) = (gtg^{-1},z)$.  This map is 
$G$-equivariant, for the right action of $G$ on $Y$ by conjugation, if we make $G$ act on the 
right of $G/T\times Y_T$ by $(gT,(t,z))\cdot h = (h^{-1}gT,(t,z))$.  By Lemma 6.3 
of \cite{MurSte}, the induced map $p_Y^*\colon \Omega^*(Y)\to \Omega^*(G/T\times Y_T)$ on forms 
is injective.  Therefore it suffices to prove that $\delta(p_Y^*(f)) = dp_Y^*(\beta) = \pi^*(p_Y^*(\omega))$ 
in $\Omega^2(G/T\times Y_T^{[2]})$.  

Recall that we may identify a point in $G/T$ with a family of orthogonal projections 
$P_1,\ldots,P_n$ where $P_iP_j = 0$ if $i\neq j$ and $\sum_i P_1 = 1$.  We identify 
a point in $G/T\times Y_T$ with a triple $(P,\lambda,z)$, where $P = (P_1,\ldots,P_n)$ 
is a family of orthogonal projections as above, $\lambda = (\lambda_1,\ldots,\lambda_n)\in T$ 
with $z\neq \lambda_i$ for all $i$.  Under this identification the right action of 
$G$ is $(P,\lambda,z)\cdot h = (h^{-1}Ph,\lambda,z)$.  We regard the $\lambda_i$ as the eigenvalues of a 
unitary matrix $g$ and the $P_i$ as the orthogonal projections onto the $\lambda_i$-eigenspace.  
Under this interpretation, the map $G/T\times Y_T\to Y$ is the map which sends 
\[
(P,\lambda,z)\mapsto (g,z),\quad \text{where}\ g = \sum^n_{i=1}\lambda_iP_i.  
\]   

From equation (B.4) in \cite{MurSte} we have the following expression for the curving $p_Y^*(f)$: 
\[
p_Y^*(f) = \frac{i}{4\pi}\sum_{i\neq k}\left(\log_z\lambda_i - \log_z\lambda_k + \frac{\lambda_k-\lambda_i}{\lambda_k}\right)
\tr(P_idP_kdP_k).  
\]
A little calculation yields that 
\[
\delta(p_Y^*(f)) = \frac{i}{4\pi}\sum_{i\neq k}A_{ik}\left(\tr(P_i[P_k,\theta_h]dP_k )+ \tr(P_idP_k[P_k,\theta_h]) + 
\tr(P_i[P_k,\theta_h][P_k,\theta_h])\right),
\]
where we have set $A_{ik} = \log_z\lambda_i - \log_z\lambda_k + (\lambda_k -\lambda_i)\lambda_i^{-1}$ 
and $\theta_h = dhh^{-1}$.  
Using the fact that $P_iP_k=0$ for $i\neq k$ and $dP_k = P_kdP_k + dP_kP_k$ we obtain 
\[
\tr(P_i[P_k,\theta_h]dP_k) = - \tr(\theta_hdP_kP_i).  
\]
Similarly we obtain 
\begin{align*} 
& \tr(P_idP_k[P_k,\theta_h]) = -\tr(\theta_h P_idP_k) \\ 
& \tr(P_i[P_k,\theta_h][P_k,\theta_h]) = -\tr(P_i\theta_hP_k\theta_h).  
\end{align*} 
Hence our expression for $\delta(p_Y^*(f))$ becomes 
\[
\delta(p_Y^*(f)) = -\frac{i}{4\pi}\sum_{i\neq k} A_{ik}\left(\tr(dP_kP_i\theta_h) + \tr(P_idP_k\theta_h) + \tr(P_i\theta_hP_k\theta_h)\right). 
\]
This splits up into the sum of two terms: 
\begin{equation} 
\label{app:eq1} 
-\frac{i}{4\pi}\sum_{i\neq k}(\log_z\lambda_i - \log_z\lambda_k)\left(\tr(\theta_hdP_kP_i) + 
\tr(\theta_hP_idP_k) + \tr(P_i\theta_hP_k\theta_h)\right) 
\end{equation} 
and 
\begin{equation} 
\label{app:eq2} 
-\frac{i}{4\pi}\sum_{i\neq k}(1-\lambda_i\lambda_k^{-1})\left(\tr(\theta_hdP_kP_i) + \tr(\theta_hP_idP_k) + \tr(P_i\theta_hP_k\theta_h)\right) .
\end{equation} 
We simplify the term~\eqref{app:eq1}.  Using the fact that $\sum_iP_i = I$ and $\sum_idP_i = 0$ we have 
\[
\sum_{i,k}\log_z\lambda_i\left(\tr(\theta_hdP_kP_i) + \tr(\theta_hP_idP_k)\right) = 0.   
\]
Therefore, 
\begin{align*} 
& \phantom{-}\sum_{i\neq k} \log_z\lambda_i\left(\tr(\theta_hdP_kP_i) + \tr(\theta_hP_idP_k)\right) \\ 
= & - \sum_i\log_z\lambda_i\left(\tr(\theta_hdP_iP_i) + \tr(\theta_hP_idP_i)\right) \\ 
= & -\sum_i\log_z\lambda_i\tr(\theta_hdP_i).  
\end{align*} 
using $dP_iP_i + P_idP_i = dP_i$.  Similarly we have 
\[
\sum_{i\neq k}\log_z\lambda_k\left(\tr(\theta_hdP_kP_i) + \tr(\theta_hP_idP_k)\right) = \sum_k\log_z\lambda_k\tr(\theta_hdP_k).  
\]
For the remaining terms in~\eqref{app:eq1} we have 
\[
\sum_{i,k}\log_z\lambda_i\tr(P_i\theta_hP_k\theta_h) = \sum_i\log_z\lambda_i\tr(P_i\theta_h\theta_h).  
\]
Hence 
\begin{align*} 
\sum_{i\neq k} \log_z\lambda_i\tr(P_i\theta_hP_k\theta_h)  & = \sum_i\log_z\lambda_i\tr(P_i\theta_h\theta_h) - 
\sum_i\log_z\lambda_i\tr(P_i\theta_hP_i\theta_h) \\ 
&= \sum_i\log_z\lambda_i\tr(P_i\theta_h\theta_h), 
\end{align*} 
since $\tr(P_i\theta_hP_i\theta_h) = 0$.  Similarly we have 
\[
\sum_{i\neq k}\log_z\lambda_k\tr(P_i\theta_hP_k\theta_h) = - \sum_k\log_z\lambda_k\tr(P_k\theta_h\theta_h) .
\]
Therefore the term~\eqref{app:eq1} reduces to 
\begin{equation}
\label{app:eq3} 
-\frac{i}{2\pi}\sum_i\log_z\lambda_i\left(\tr(P_i\theta_h\theta_h) - \tr(\theta_hdP_i)\right) .
\end{equation}   
We compare the term~\eqref{app:eq3} with $dp_Y^*\beta$.  We have from~\eqref{eq:eqn for beta} 
\[
p_Y^*\beta = -\frac{i}{2\pi}\sum_i\log_z\lambda_i \tr(\theta_hP_i) 
\]
and hence 
\begin{equation}
\label{app:eq4}
dp_Y^*(\beta)  = -\frac{i}{2\pi} \sum_i\log_z\lambda_i\left(\tr(-\theta_hdP_i) + \tr(\theta_h\theta_hP_i)\right) 
- \frac{i}{2\pi}\sum_i\frac{d\lambda_i}{\lambda_i}\tr(\theta_hP_i)
\end{equation} 
Comparing~\eqref{app:eq3} and~\eqref{app:eq4} we obtain the following expression for $\delta(p_Y^*(f))-dp_Y^*(\beta)$: 
\begin{multline} 
\label{app:eq5}
\delta(p_Y^*(f))-dp_Y^*(\beta)  = -\frac{i}{4\pi}\sum_{i\neq k}(1-\lambda_i\lambda_k^{-1})\left[ 
\tr(\theta_hdP_kP_i) + \tr(\theta_hP_idP_k) + \tr(P_i\theta_hP_k\theta_h)\right] \\
+ \frac{i}{2\pi}\sum_i\frac{d\lambda_i}{\lambda_i}\tr(\theta_hP_i) .
\end{multline}
We have, using $\sum_kdP_k = 0$, 
\begin{align*} 
& \phantom{-} \sum_{i\neq k} \tr(\theta_hdP_kP_i) + \tr(\theta_hP_idP_k) \\ 
& = - \sum_i\tr(\theta_hdP_iP_i) + \tr(\theta_hP_idP_i) \\ 
& = -\sum_i\tr(\theta_hdP_i) \\ 
& = 0 .
\end{align*} 
Similarly, using $\sum P_k = I$ and $\tr(P_i\theta_hP_i\theta_h) = 0$, we have 
\[
\sum_{i\neq k} \tr(P_i\theta_hP_k\theta_h) = \tr(\theta_h\theta_h) = 0.  
\]
Therefore the expression for $\delta(p_Y^*(f)) - dp_Y^*(\beta)$ in~\eqref{app:eq5} reduces to 
\begin{multline} 
\label{app:eq6} 
\delta(p_Y^*(f)) - dp_Y^*(\beta) = +\frac{i}{4\pi}\sum_{i\neq k} \lambda_i\lambda_k^{-1}\left[
\tr(\theta_hdP_kP_i) + \tr(\theta_hP_idP_k) + \tr(P_i\theta_hP_k\theta_h)\right] \\ 
+ \frac{i}{2\pi} \sum_i \frac{d\lambda_i}{\lambda_i}\tr(\theta_hP_i).  
\end{multline} 
For the term 
\[
\sum_{i\neq k} \lambda_i\lambda_k^{-1}\tr(P_i\theta_hP_k\theta_h) 
\]
appearing in~\eqref{app:eq6} we have, since $\tr(P_i\theta_hP_i\theta_h) = 0$,  
\begin{align*} 
\sum_{i\neq k}\lambda_i\lambda_k^{-1}\tr(P_i\theta_hP_k\theta_h) & = 
\sum_{i,k}\lambda_i\lambda_k^{-1}\tr(P_i\theta_hP_k\theta_h) \\ 
& =\tr(g\theta_hg^{-1}\theta_h) \\ 
& = - \tr(\theta_h\hat{\theta}_h) ,
\end{align*}
where we have set $\hat{\theta}_h = g^{-1}\theta_hg$.  For the term 
\[
\sum_{i\neq k} \lambda_i\lambda_k^{-1}\left[\tr(\theta_hdP_kP_i) + \tr(\theta_hP_idP_k)\right] 
\]
we have, using $dP_iP_i + P_idP_i = dP_i$ and $\sum_i dP_i = 0$, 
\[
\sum_{i\neq k} \lambda_i\lambda_k^{-1}\left[\tr(\theta_hdP_kP_i) + \tr(\theta_hP_idP_k)\right] 
= \sum_k \lambda_k^{-1}\left[\tr(\theta_hdP_k g) + \tr(\theta_hgdP_k)\right] .
\]
Therefore~\eqref{app:eq6} becomes 
\begin{multline*} 
\delta(p_Y^*f) -dp_Y^*\beta = \frac{i}{4\pi}\left\{ -\tr(\theta_h\hat{\theta}_h) + \sum_k \lambda_k^{-1}\left[ 
\tr(\theta_hdP_kg ) + \tr(\theta_hgdP_k)\right] \right. \\ 
\left. - 2\sum_kd\lambda_k \lambda_k^{-1}\tr(\theta_hP_k)\right\}  .
\end{multline*} 
We have, using $P_k^2 = P_k$, 
\begin{align*} 
& \sum_k \left(\lambda_k^{-1}\tr(\theta_hdP_k g) - d\lambda_k \lambda_k^{-1}\tr(\theta_hP_k)\right) \\ 
= & \sum_k \lambda_k^{-1}\tr(\theta_hdP_k g) - \lambda_k^{-1}d\lambda_k \lambda_k^{-1} \tr(\theta_hP_k g) \\ 
= & -\tr(\theta_hg^{-1}dg) \\ 
= & -\tr( \theta_h\theta), 
\end{align*} 
using $d(g^{-1}) = \sum_k (\lambda_k^{-1}dP_k -  \lambda_k^{-1}d\lambda_k\lambda_k^{-1}P_k)$, where 
we have set $\theta = g^{-1}dg$.  Similarly, 
\[
\sum_k \left(\lambda_k^{-1}\tr(\theta_hgdP_k) - d\lambda_k \lambda_k^{-1}\tr(\theta_hP_k)\right) = - \tr(\hat{\theta}_h\theta) .
\]
Hence 
\[
\delta(p_Y^*(f)) - dp_Y^*(\beta) = p_Y^*\pi^*\left(\frac{i}{4\pi} \left(\tr(\hat{\theta}_h\theta_h) + 
\tr(\theta \theta_h) + \tr(\theta \hat{\theta}_h)\right)\right).
\]


\end{document}